\documentclass[12pt]{article}
\usepackage{authblk}
\usepackage[utf8]{inputenc}
\usepackage{amsmath}
\usepackage{amsthm}
\usepackage{amscd}
\usepackage[all]{xy}
\xyoption{2cell}
\UseAllTwocells
\usepackage{enumitem}
\usepackage{geometry}
\usepackage[charter]{mathdesign}
\usepackage{enumitem}

\newtheorem{theorem}			     {Theorem}	    [section]
\newtheorem{proposition}  [theorem]	 {Proposition}
\newtheorem{corollary}	  [theorem]	 {Corollary}

\theoremstyle{definition}

\newtheorem{remark} 	  [theorem]  {Remark}

\newtheorem{openquestion}	  [theorem]  {Open question}

\newcommand{\CC}{\mathbb{C}}
\newcommand{\DD}{\mathbb{D}}
\newcommand{\e}{\mathrm{e}}
\newcommand{\F}{\mathrm{F}}
\newcommand{\JJ}{\mathbb{J}}
\newcommand{\M}{\mathrm{M}}
\newcommand{\Mc}{\mathcal{M}}
\newcommand{\N}{\mathrm{N}}
\renewcommand{\P}{\mathrm{P}}
\newcommand{\Q}{\mathrm{Q}}
\newcommand{\T}{\mathrm{T}}
\newcommand{\U}{\mathrm{U}}
\newcommand{\VV}{\mathbb{V}}
\newcommand{\Y}{\mathrm{Y}}
\newcommand{\matr}{\mathrm{matr}}
\newcommand{\mclex}{\mathrm{mclex}}
\newcommand{\all}{\mathrm{all}}
\newcommand{\strong}{\mathrm{strong}}
\newcommand{\lex}{\mathrm{lex}}
\newcommand{\reg}{\mathrm{reg}}
\newcommand{\essalg}{\mathrm{ess\,\, alg}}
\newcommand{\Mal}{\mathrm{Mal}}
\newcommand{\Maj}{\mathrm{Maj}}
\newcommand{\Ari}{\mathrm{Ari}}
\newcommand{\Uni}{\mathrm{Uni}}
\newcommand{\StrUni}{\mathrm{StrUni}}
\newcommand{\Sub}{\mathrm{Sub}}
\newcommand{\Cube}{\mathrm{Cube}}
\newcommand{\Edge}{\mathrm{Edge}}
\newcommand{\Set}{\mathsf{Set}}
\newcommand{\Part}{\mathsf{Part}}
\newcommand{\Bool}{\mathsf{Bool}}
\renewcommand{\1}{\mathsf{1}}
\newcommand{\op}{\mathrm{op}}
\newcommand{\pb}[1][dr]{\save*!/#1-1.5pc/#1:(-1,1)@^{|-}\restore}
\newdir{ >}{{}*!/-8pt/@{>}}

\title{Partial algebras and implications of (weak) matrix properties}
\author[a,d,e]{Michael Hoefnagel}
\author[b,c,d]{Pierre-Alain Jacqmin\thanks{The second
author is grateful to the FNRS for its generous support.}}
\affil[a]{\small{\textit{Mathematics Division, Department of Mathematical Sciences, Stellenbosch University, Private Bag X1 Matieland 7602, South Africa}}}
\affil[b]{\small{\textit{Institut de Recherche en Math\'ematique et Physique, Universit\'e catholique de Louvain, Chemin du Cyclotron~2, B 1348 Louvain-la-Neuve, Belgium}}}
\affil[c]{\small{\textit{Department of Mathematics, Royal Military Academy, Brussels, Belgium\vspace{5pt}}}}
\affil[d]{\small{\textit{Centre for Experimental Mathematics, Department of Mathematical Sciences, Stellenbosch University, Private Bag X1 Matieland 7602, South Africa}}}
\affil[e]{\small{\textit{National Institute for Theoretical and Computational Sciences (NITheCS), South Africa}}}
\date{}

\begin{document}

\maketitle

\begin{abstract}
Matrix properties are a type of property of categories which includes the ones of being Mal'tsev, arithmetical, majority, unital, strongly unital and subtractive. Recently, an algorithm has been developed to determine implications $\M\Rightarrow_{\lex_\ast}\N$ between them. We show here that this algorithm reduces to construct a partial term corresponding to $\N$ from a partial term corresponding to~$\M$. Moreover, we prove that this is further equivalent to the corresponding implication between the weak versions of these properties, i.e., the one where only strong monomorphisms are considered instead of all monomorphisms.
\end{abstract}



\section*{Introduction}

A pointed variety $\VV$ of universal algebras in said to be  \emph{subtractive} in the sense of~\cite{Ursini1994} if its theory contains a binary term $x-y$ (called a \emph{subtraction}) satisfying the equations
\begin{equation}\label{equation subtraction}
\begin{cases}x-0=x\\ x-x=0\end{cases}
\end{equation}
where $0$ is the unique constant in the theory of~$\VV$. From such a term, one can construct a ternary term $p(x,y,z)$ satisfying the equations
\begin{equation}\label{equation term p intro}
\begin{cases}p(x,x,x)=x\\ p(x,x,0)=0\\ p(0,x,x)=0\end{cases}
\end{equation}
by setting
$$p(x,y,z)=x-(y-z).$$
Indeed, one can check that
\begin{equation}\label{equations p from -}
\begin{cases}p(x,x,x)=x-(x-x)=x-0=x\\p(x,x,0)=x-(x-0)=x-x=0\\p(0,x,x)=0-(x-x)=0-0=0.\end{cases}
\end{equation}
In addition, from a subtraction $x-y$, one can also construct a $4$-ary term $q(x,y,z,w)$ satisfying the equations
\begin{equation}\label{equation term q intro}
\begin{cases}q(x,0,0,0)=x\\q(x,x,y,y)=0\\q(x,y,x,y)=0\end{cases}
\end{equation}
by setting
$$q(x,y,z,w)=(x-y)-(z-w).$$
To see this, it suffices to compute
\begin{equation}\label{equations q from -}
\begin{cases}q(x,0,0,0)=(x-0)-(0-0)=x-0=x\\q(x,x,y,y)=(x-x)-(y-y)=0-0=0\\q(x,y,x,y)=(x-y)-(x-y)=0.\end{cases}
\end{equation}
Although these two proofs look very similar, let us now explain a fundamental difference between them which has deep consequences in Category Theory.

The first proof (the construction of the term $p(x,y,z)$ from the subtraction $x-y$) can be extended to the context of partial algebras in the following sense, while the second cannot. Let us consider a pointed set $(A,0)$ equipped with a \emph{partial subtraction}, i.e., a partial function $-\colon A^2\dashrightarrow A$ such that the existential equations
$$\begin{cases}x-0=^\e x\\ x-x=^\e 0\end{cases}$$
hold, meaning that, for every $a\in A$, both $a-0$ and $a-a$ are defined and they are equal to $a-0=a$ and $a-a=0$ respectively. For such a \emph{partial subtractive algebra}~$A$, one can construct the ternary partial operation $p\colon A^3\dashrightarrow A$ by setting $p(x,y,z)=x-(y-z)$ for $x,y,z\in A$ for which $y-z$ and $x-(y-z)$ are defined (and $p(x,y,z)$ is not defined for the other triples $(x,y,z)\in A^3$). This partial operation satisfies the existential equations
$$\begin{cases}p(x,x,x)=^\e x\\ p(x,x,0)=^\e 0\\ p(0,x,x)=^\e 0\end{cases}$$
which can be seen using a proof completely analogous to the one in~(\ref{equations p from -}), with the extra care of checking that every expression which appears is indeed defined.

For the second proof (i.e., the construction of $q(x,y,z,w)$ from the subtraction $x-y$), this is no longer the case. One can still define the $4$-ary partial operation $q\colon A^4\dashrightarrow A$ by setting $q(x,y,z,w)=(x-y)-(z-w)$ exactly for those $x,y,z,w\in A$ for which $x-y$, $z-w$ and $(x-y)-(z-w)$ are defined, but now the existential equation
$$q(x,y,x,y)=^\e 0$$
can no longer be shown to hold in $A$ analogously as in~(\ref{equations q from -}) since $x-y$ is in general not defined for all $x,y\in A$.

Let us now explain the consequences that this difference implies in Category Theory. Let us first see how the algebraic condition that a pointed variety of universal algebras has a subtraction in its theory extends to the context of a finitely complete pointed category. The first thing one can think of is to require that any object $X$ is naturally equipped with an internal subtraction $s\colon X^2\to X$ satisfying the internal version of the axioms $s(x,0)=x$ and $s(x,x)=0$. However, this is in general too strong of a condition; even the category of \emph{subtractive algebras} (i.e., pointed sets equipped with a subtraction~$-$) does not have such an internal operation since $-\colon A^2\to A$ is not a homomorphism in general. The `right' approach has been proposed in~\cite{Janelidze2005} and generalized in~\cite{Janelidze2006} to all Mal'tsev conditions of a similar form to those appearing in~(\ref{equation subtraction}), (\ref{equation term p intro}) and~(\ref{equation term q intro}). The idea is, instead of looking at the equations~(\ref{equation subtraction}) line by line, to look at them `components by components'. In other words, if we write the equations~(\ref{equation subtraction}) in the form of an (extended) matrix
$$\Sub=\left[\begin{array}{cc|c}
x & 0 & x \\
x & x & 0
\end{array}\right],$$
the columns of this matrix give a condition on the relations in a pointed variety~$\VV$. Indeed, writing elements of $A^2$ as column vectors, the `matrix condition' corresponding to $\Sub$ requires that, for any homomorphic binary relation $R\subseteq A^2$ in~$\VV$, we have
$$\left\{\left[\begin{array}{c}
a  \\
a
\end{array}\right], \,\left[\begin{array}{c}
0 \\
a
\end{array}\right] \right\} \subseteq R \quad \implies \quad \left[\begin{array}{c}
a  \\
0
\end{array}\right] \in R$$
for every $a\in A$. Using generalized elements, this condition can be extended to the context of finitely complete pointed categories. Such a category for which each binary relation satisfies this condition is called a \emph{subtractive category}~\cite{Janelidze2005}. Using a similar technique introduced in~\cite{Janelidze2006}, we can extend the conditions of having a ternary term $p$ satisfying the equations~(\ref{equation term p intro}) and of having a $4$-ary term $q$ satisfying the equations~(\ref{equation term q intro}) from the pointed varietal context to the finitely complete pointed context using the matrices
$$\P=\left[\begin{array}{ccc|c}
x & x & x & x \\
x & x & 0 & 0 \\
0 & x & x & 0
\end{array}\right]$$
and
$$\Q=\left[\begin{array}{cccc|c}
x & 0 & 0 & 0 & x \\
x & x & y & y & 0 \\
x & y & x & y & 0
\end{array}\right]$$
respectively. The fact that one can construct a partial term $p$ satisfying the equations~(\ref{equation term p intro}) from a partial subtraction is equivalent to the fact that each subtractive category satisfies the property corresponding to the matrix~$P$, which we denote by $\Sub\Rightarrow_{\lex_\ast} \P$. This equivalence is part of our main Theorem~\ref{theorem implication}. In addition, still according to this theorem, this implication is also equivalent to the implication where the weak versions of these properties are involved, that is, the properties only requiring that the strong relations satisfy the matrix properties. Strong relations (or in general strong monomorphisms) are those which are orthogonal to epimorphisms. We denote this implication by $\Sub\Rightarrow_{\lex_\ast}^\strong\P$, meaning that any \emph{weakly subtractive category} has strong $\P$-closed relations (i.e., has weakly the property corresponding to~$\P$).

Besides, as we noted above, the proof that any subtractive variety admits in its theory a $4$-ary term $q$ satisfying the equations~(\ref{equation term q intro}) cannot be extended to the partial algebraic context. We can actually show that, even using another construction, it is not possible to construct a partial term $q$ satisfying the equations~(\ref{equation term q intro}) existentially from a partial subtraction. In view of our Theorem~\ref{theorem implication}, this thus means that the implications $\Sub\Rightarrow_{\lex_\ast}\Q$ and $\Sub\Rightarrow_{\lex_\ast}^\strong\Q$ do \emph{not} hold.

The proof of Theorem~\ref{theorem implication} relies on two main ingredients. On one hand, we make use of the embedding theorem from~\cite{Jacqmin2019} in which each finitely complete pointed category satisfying the weak property corresponding to an extended matrix $\M$ as above is fully faithfully embedded in a power of the category of partial algebras determined by~$\M$. On the other hand, we use the algorithms developed in~\cite{HJ2022,HJJ2022} to decide whether an implication $\M\Rightarrow_{\lex_\ast}\N$ holds for extended matrices $\M$ and $\N$ as above. Using a computer implementation of these algorithms, it has been shown in~\cite{HJ2022} that the implication $\Sub\Rightarrow_{\lex_\ast}\Q$ does indeed \emph{not} hold (see Figure~8 in that paper), proving in view of our Theorem~\ref{theorem implication} that a partial $4$-ary term $q$ satisfying the equations~(\ref{equation term q intro}) existentially cannot be constructed from a partial subtraction.

As classical examples of properties induced by matrices, let us mention the ones of being \emph{subtractive}~\cite{Janelidze2005}, \emph{unital}~\cite{Bourn1996} and \emph{strongly unital}~\cite{Bourn1996}. Since our results are also formulated in the non-pointed context, let us also mention the properties of being \emph{Mal'tsev}~\cite{CLP1991,CPP1992}, \emph{majority}~\cite{Hoefnagel2019} and \emph{arithmetical}~\cite{Pedicchio1996,HJJ2022}. Among their weak versions, the ones of being \emph{weakly Mal'tsev}~\cite{MartinsFerreiraPhD,MartinsFerreira2008} and \emph{weakly unital}~\cite{MartinsFerreiraPhD} have already been studied.

The rest of the paper is divided in two sections. In the first one, we essentially recall the needed concepts and results from the literature. Among others, we remind the reader with the matrix properties and the algorithms from~\cite{HJ2022,HJJ2022} to decide implications between them. We also briefly recall the theory of partial algebras and the embedding theorem established in~\cite{Jacqmin2019}. In the second section, we first extend the characterizations found in the literature of \emph{trivial matrices} (i.e., those inducing properties that can be satisfied only by preorders) using the language of partial algebras (Theorem~\ref{theorem trivial}). We then come to our main theorem (Theorem~\ref{theorem implication}) characterizing when a matrix property (or rather a set of matrix properties) implies another one. As explained above, we prove that the implication $\M\Rightarrow_{\lex_\ast}\N$ is equivalent to $\M\Rightarrow_{\lex_\ast}^\strong\N$ and also to the condition that one can construct a partial term corresponding to $\N$ from one corresponding to~$\M$. A similar characterization is also established in the non-pointed context. As a corollary (Corollary~\ref{corollary anti-trivial}), we also give a characterization of \emph{anti-trivial matrices}, i.e., those inducing a property satisfied by any finitely complete pointed category. Let us finally mention that in all these characterizations, we also consider the property on a finitely complete (pointed) category induced by a matrix $\M$ in which only relations in a specified class of monomorphisms $\Mc$ are required to be $\M$-closed. This class $\Mc$ is asked to be stable under pullbacks, closed under composition and containing the regular monomorphisms.

\section{Preliminaries}\label{section preliminaries}

\subsection*{Matrix properties}

Given integers $n>0$ and $m,k\geqslant 0$, we denote as in~\cite{HJ2022} by $\matr_\ast(n,m,k)$ the set of (extended) matrices
$$\M=\left[\begin{array}{ccc|c} x_{11} & \dots & x_{1m} & x_{1\,m+1} \\ \vdots & & \vdots & \vdots \\ x_{n1} & \dots & x_{nm} & x_{n\,m+1} \end{array}\right]$$
with $n$ rows, $m$ left columns, one right column and whose entries lie in the set $\{\ast,x_1,\dots,x_k\}$ (i.e., the free pointed set on $k$ variables $x_1,\dots,x_k$). Each such matrix $\M\in\matr_\ast(n,m,k)$ determines a property on $n$-ary (internal) relations $r\colon R\rightarrowtail X_1\times \cdots \times X_n$ in finitely complete pointed categories~$\CC$. Such a relation is given by a subobject of the $n$-fold product $X_1\times \cdots \times X_n$ and is represented by a monomorphism~$r$. To describe the property determined by $\M$ on such relations, we need the concept of row-wise interpretations. Given pointed sets $(S_1,\ast_1),\dots,(S_n,\ast_n)$, a \emph{row-wise interpretation} of $\M$ of type $((S_1,\ast_1),\dots,(S_n,\ast_n))$ is an extended matrix
$$\left[\begin{array}{ccc|c} f_1(x_{11}) & \dots & f_1(x_{1m}) & f_1(x_{1\,m+1}) \\ \vdots & & \vdots & \vdots \\ f_n(x_{n1}) & \dots & f_n(x_{nm}) & f_n(x_{n\,m+1}) \end{array}\right]$$
whose entries are obtained by applying to the elements of the $i^{\textrm{th}}$ row of~$\M$ (for each $i\in\{1,\dots,n\}$) a \emph{pointed function} $f_i\colon(\{\ast,x_1,\dots,x_k\},\ast)\to (S_i,\ast_i)$, i.e., a function $f_i\colon\{\ast,x_1,\dots,x_k\}\to S_i$ such that $f_i(\ast)=\ast_i$. Viewing each hom-set $\CC(X,Y)$ of a finitely complete pointed category $\CC$ as a pointed set where the distinguished element is the zero morphism $0\colon X\to Y$, we say that an $n$-ary relation $r\colon R\rightarrowtail X_1\times \cdots \times X_n$ in $\CC$ is \emph{strictly $\M$-closed}~\cite{Janelidze2006} if, for any object $Z$ of~$\CC$, each row-wise interpretation
$$\left[\begin{array}{ccc|c} g_{11} & \dots & g_{1m} & h_1 \\ \vdots & & \vdots & \vdots \\ g_{n1} & \dots & g_{nm} & h_n \end{array}\right]$$
of $\M$ of type $(\CC(Z,X_1),\dots,\CC(Z,X_n))$ is \emph{compatible} with~$r$, i.e., if, for each $j\in\{1,\dots,m\}$, the morphism $(g_{1j},\dots,g_{nj})\colon Z\to X_1\times\cdots\times X_n$ induced by the $j^{\textrm{th}}$ left column factors through~$r$, then the morphism $(h_1,\dots,h_n)\colon Z\to X_1\times\cdots\times X_n$ induced by the right column also factors through~$r$.

For a pointed set $(S,\ast)$, a (non-row-wise) \emph{interpretation} of $\M$ of type $(S,\ast)$ is a row-wise interpretation of $\M$ of type $((S,\ast),\dots,(S,\ast))$ for which $f_1=\dots=f_n$. We say that a relation $r\colon R\rightarrowtail X^n$ in~$\CC$ is \emph{$\M$-closed}~\cite{Janelidze2006} if, for any object $Z$ of~$\CC$, each interpretation of $\M$ of type $\CC(Z,X)$ is compatible with~$r$.

We denote by $\matr(n,m,k)$ the subset of \emph{non-pointed matrices} of $\matr_\ast(n,m,k)$, i.e., those matrices not containing $\ast$ as an entry. For a non-pointed matrix $\M\in\matr(n,m,k)$, we define analogously as above the concepts of row-wise interpretations of $\M$ of type $(S_1,\dots,S_n)$ and interpretations of $\M$ of type $S$ where $S_1,\dots,S_n,S$ are mere sets. This gives rise to the concepts of \emph{strictly $\M$-closed relations} $r\colon R\rightarrowtail X_1\times\cdots\times X_n$ and of \emph{$\M$-closed relations} $r\colon R\rightarrowtail X^n$ in a (non necessarily pointed) finitely complete category~$\CC$. We also denote by $\matr$ and $\matr_\ast$ the unions
$$\matr=\bigcup_{\substack{n > 0\\m \geqslant 0\\k\geqslant 0}} \matr(n,m,k) \qquad\text{and}\qquad \matr_\ast=\bigcup_{\substack{n > 0\\m \geqslant 0\\k\geqslant 0}} \matr_\ast(n,m,k).$$

Given a class of monomorphisms $\Mc$ stable under pullbacks in a finitely complete category~$\CC$ (respectively in a finitely complete pointed category~$\CC$) and a matrix $\M\in\matr(n,m,k)$ (respectively $\M\in\matr_\ast(n,m,k)$), we say that $\CC$ \emph{has $\M$-closed $\Mc$-relations} if any relation $r\colon R\rightarrowtail X^n$ in $\Mc$ is $\M$-closed (since $\Mc$ is stable under pullbacks, it is in particular closed under pre-composition with isomorphisms and therefore, a relation is said to be in $\Mc$ if one (thus all) representing monomorphism of the relation is in~$\Mc$). Using (the internal version of) Proposition~1.9 in~\cite{Janelidze2006}, we immediately get the following proposition, which slightly generalizes Theorem 2.4 of that paper.

\begin{proposition}\label{proposition M closed and strict M closed}
Let $n>0$ and $m,k\geqslant 0$ be integers, $\M\in\matr(n,m,k)$ (respectively $\M\in\matr_\ast(n,m,k)$), $\CC$ a finitely complete category (respectively a finitely complete pointed category) and $\Mc$ a class of monomorphisms in $\CC$ stable under pullbacks. Then the following statement are equivalent:
\begin{itemize}
\item $\CC$ has $\M$-closed $\Mc$-relations, i.e., any relation $r\colon R\rightarrowtail X^n$ in $\Mc$ is $\M$-closed;
\item $\CC$ has strictly $\M$-closed $\Mc$-relations, i.e., any relation $r\colon R\rightarrowtail X_1\times\cdots\times X_n$ in $\Mc$ is strictly $\M$-closed.
\end{itemize}
\end{proposition}

If one takes $\Mc$ to be the class $\Mc_{\all}$ of all monomorphisms, one gets back the notion of (pointed) finitely complete categories with $\M$-closed relations introduced in~\cite{Janelidze2006} and extensively studied in \cite{HJ2022,HJ2023,HJ2023b,HJJ2022,HJJW2023,Janelidze2006b}. Before giving some examples, let us recall the characterization of (finitary one-sorted) varieties of universal algebras with $\M$-closed relations. We first remind the reader that a variety $\VV$ is a pointed category if and only if it has a nullary term $0$ and any two nullary terms are equal in its theory. Note that, although the notion of $\M$-closed relations is defined via the columns of~$\M$, the characterization of varieties with $\M$-closed relations reads each row of the matrix as an equation.

\begin{theorem}\label{theorem characterization varieties with M-closed relations}\cite{Janelidze2006}
Let $n>0$ and $m,k\geqslant 0$ be integers and
$$\M=\left[\begin{array}{ccc|c} x_{11} & \dots & x_{1m} & x_{1\,m+1} \\ \vdots & & \vdots & \vdots \\ x_{n1} & \dots & x_{nm} & x_{n\,m+1} \end{array}\right]$$
be a matrix in $\matr(n,m,k)$ (respectively in $\matr_\ast(n,m,k)$), whose entries thus lie in $\{x_1,\dots,x_k\}$ (respectively in $\{\ast,x_1,\dots,x_k\}$). A variety $\VV$ (respectively a pointed variety~$\VV$) has $\M$-closed ($\Mc_\all$\nobreakdash-)rel\-ations if and only if its theory contains an $m$-ary term $p$ such that, for each $i\in\{1,\dots,n\}$, the equation
$$p(x_{i1},\dots,x_{im})=x_{i\,m+1}$$
holds in~$\VV$ (where $x_1,\dots,x_k$ are interpreted as different variables and, in the pointed case, $\ast$ is interpreted as the unique constant of~$\VV$).
\end{theorem}

\subsection*{Examples}

We first focus on examples where $\Mc=\Mc_\all$. For the matrix
$$\Mal=\left[\begin{array}{ccc|c}
x_1 & x_2 & x_2 & x_1 \\
x_2 & x_2 & x_1 & x_1
\end{array}\right]$$
in $\matr(2,3,2)$, a finitely complete category with $\Mal$-closed ($\Mc_\all$-)relations is a \emph{Mal'tsev category}~\cite{CPP1992}, generalizing the notions of regular Mal'tsev categories~\cite{CLP1991} and of Mal'tsev varieties~\cite{Maltsev1954} to the left exact context. Making explicit the definition for this matrix, it turns out that a finitely complete category is a Mal'tsev category (i.e., has $\Mal$-closed relations) if and only if, for any binary relation $r\colon R \rightarrowtail X^2$ and any pair of morphisms $z_1,z_2\colon Z\to X$, if the three induced morphisms $(z_1,z_2)$, $(z_2,z_2)$ and $(z_2,z_1)\colon Z\to X^2$ factor through~$r$, so does the morphism $(z_1,z_1)\colon Z\to X^2$. Using Proposition~\ref{proposition M closed and strict M closed}, this is equivalent to the condition that any binary relation $r\colon R\rightarrowtail X_1\times X_2$ is \emph{difunctional}~\cite{Riguet1948}, i.e., given morphisms $z_1,z_2\colon Z\to X_1$ and $z'_1,z'_2\colon Z\to X_2$ such that the morphisms $(z_1,z'_2)$, $(z_2,z'_2)$ and $(z_2,z'_1)\colon Z\to X_1\times X_2$ factor through~$r$, so does the morphism $(z_1,z'_1)\colon Z\to X_1\times X_2$. In this case, Theorem~\ref{theorem characterization varieties with M-closed relations} gives back Mal'tsev's theorem~\cite{Maltsev1954}, i.e., that  a variety $\VV$ of universal algebras is a Mal'tsev category if and only if there exists a ternary term $p(x,y,z)$ in the theory of $\VV$ such that the identities
$$\begin{cases}p(x_1,x_2,x_2) = x_1 \\ p(x_2,x_2,x_1) = x_1 \end{cases}$$
hold in~$\VV$.

If we denote by $\Maj$ the matrix
$$\Maj=\left[\begin{array}{ccc|c}
x_1 & x_1 & x_2 & x_1 \\
x_1 & x_2 & x_1 & x_1 \\
x_2 & x_1 & x_1 & x_1
\end{array}\right]$$
in $\matr(3,3,2)$, we get the notion of \emph{majority categories} as introduced in~\cite{Hoefnagel2019}, i.e., finitely complete categories with $\Maj$-closed relations. A category which is both a majority category and a Mal'tsev category is called an \emph{arithmetical category} in~\cite{HJJ2022}, extending the terminology of~\cite{Pedicchio1996}. These are finitely complete categories with $\Ari$-closed relations, where $\Ari\in\matr(3,3,2)$ is the matrix
$$\Ari=\left[\begin{array}{ccc|c}
x_1 & x_2 & x_2 & x_1 \\
x_2 & x_2 & x_1 & x_1 \\
x_1 & x_2 & x_1 & x_1
\end{array}\right].$$

In the pointed context, we consider the matrices
$$\Uni=\left[\begin{array}{cc|c}
x_1  & \ast & x_1 \\
\ast & x_1  & x_1
\end{array}\right]$$
in $\matr_\ast(2,2,1)$ and
$$\StrUni=\left[\begin{array}{ccc|c}
x_1 & \ast & \ast & x_1 \\
x_2 & x_2  & x_1  & x_1
\end{array}\right]$$
in $\matr_\ast(2,3,2)$. A finitely complete pointed category with $\Uni$-closed relations (respectively with $\StrUni$-closed relations) is a \emph{unital category} (respectively a \emph{strongly unital category}), both notions being introduced in~\cite{Bourn1996}. One can also consider the matrix
$$\Sub=\left[\begin{array}{cc|c}
x_1 & \ast & x_1  \\
x_1 & x_1  & \ast
\end{array}\right]$$
in $\matr_\ast(2,2,1)$. \emph{Subtractive categories} (i.e., finitely complete pointed categories with $\Sub$-closed relations) have been introduced in~\cite{Janelidze2005} where it is shown that a unital subtractive category is nothing else than a strongly unital category. In order to again illustrate the definitions of categories with $\M$-closed relations, let us make explicit the definition of subtractive categories here. A finitely complete pointed category is subtractive if and only if, for each binary relation $r\colon R \rightarrowtail X^2$ and each morphism $z\colon Z\to X$, if the two induced morphisms $(z,z)$ and $(0,z)\colon Z\to X^2$ factor through~$r$, so does the morphism $(z,0)\colon Z\to X^2$. Using Proposition~\ref{proposition M closed and strict M closed}, this is equivalent to the condition that for each binary relation $r\colon R\rightarrowtail X_1\times X_2$ and every morphisms $z\colon Z\to X_1$ and $z'\colon Z\to X_2$, if the morphisms $(z,z')$ and $(0,z')\colon Z\to X_1\times X_2$ factor through~$r$, so does the morphism $(z,0)\colon Z\to X_1 \times X_2$. In this case, Theorem~\ref{theorem characterization varieties with M-closed relations} gives us that a pointed variety $\VV$ is a subtractive category if and only if it is a subtractive variety in the sense of~\cite{Ursini1994}, i.e., that its theory contains a binary term $s(x,y)$ such that the identities
$$\begin{cases}s(x,0)=x \\ s(x,x)=0 \end{cases}$$
hold in its theory (where $0$ is the unique constant of~$\VV$).

Another class of monomorphisms $\Mc$ in a category $\CC$ of interest for this paper is the class $\Mc_\strong$ of \emph{strong monomorphisms}, i.e., those monomorphisms $m\colon X\rightarrowtail Y$ which are orthogonal to every epimorphism, meaning that for each commutative square made of plain arrows
$$\xymatrix{A \ar@{->>}[r]^-{e} \ar[d]_-{f} & B \ar[d]^-{g} \ar@{.>}[ld]^-{d} \\ X \ar@{ >->}[r]_-{m} & Y}$$
where the top morphism $e\colon A\twoheadrightarrow B$ is an epimorphism, there is a (unique) dotted diagonal morphism $d\colon B\to X$ retaining the commutativity of the diagram. This class $\Mc_\strong$ of strong monomorphisms is stable under pullbacks. As it will be of importance later, let us also mention that it is closed under composition and that it contains \emph{regular monomorphisms}, i.e., equalizers. Categories with $\M$-closed $\Mc$-relations have already been considered for $\Mc=\Mc_\strong$ and a general $\M$ in~\cite{Jacqmin2019} (and implicitly, even for a general~$\Mc$).

\emph{Weakly Mal'tsev categories} have been introduced in~\cite{MartinsFerreiraPhD,MartinsFerreira2008} and characterized in~\cite{JMF2012}. In the finitely complete context and using the terminology of the present paper, they can be described as finitely complete categories with $\Mal$-closed $\Mc_\strong$-relations. They thus form a natural weakening of the notion of Mal'tsev categories. We can also see this weakening as follows. From~\cite{Bourn1996}, we know that a finitely complete category $\CC$ is a Mal'tsev category if and only if, for all pullbacks
$$\xymatrix@R=3pc@C=3pc{P \pb \ar@<2pt>@{->>}[r]^-{p_Y} \ar@<-2pt>@{->>}[d]_-{p_X} & Y \ar@{->>}@<-2pt>[d]_-{g} \ar@{ >->}@<2pt>[l]^-{r_Y} \\ X \ar@{->>}@<2pt>[r]^-{f} \ar@{ >->}@<-2pt>[u]_-{l_X} & Z \ar@{ >->}@<2pt>[l]^-{s} \ar@{ >->}@<-2pt>[u]_-{t}}$$
of split epimorphisms where $fs=1_Z=gt$, the induced morphisms $l_X=(1_X,tf)$ and $r_Y=(sg,1_Y)$ are \emph{jointly strongly epimorphic}, i.e., they do not factorize though a common proper subobject of~$P$. Besides, a finitely complete category is weakly Mal'tsev if and only if for all such pullbacks, the morphisms $l_X$ and $r_Y$ are \emph{jointly epimorphic}, i.e., for any pair of morphisms $u,v\colon P \to W$ such that $ul_X=vl_X$ and $ur_Y=vr_Y$ one has $u=v$. The characterization of varieties which form weakly Mal'tsev categories also gives rise to a Mal'tsev condition but more complicated than in the case $\Mc=\Mc_\all$. It has recently been obtained in~\cite{EJMF2023}.

\emph{Weakly unital categories} have been introduced in~\cite{MartinsFerreiraPhD} as (finitely complete) pointed categories such that, for any two objects $X,Y$, the induced morphisms $(1_X,0)\colon X \rightarrowtail X\times Y$ and $(0,1_Y)\colon Y \rightarrowtail X\times Y$ are jointly epimorphic.
$$\xymatrix@C=3pc{X \ar@{ >->}@<-2pt>[r]_-{(1_X,0)} & X \times Y \ar@{->>}@<2pt>[r]^-{p_Y} \ar@{->>}@<-2pt>[l]_-{p_X} & Y \ar@{ >->}@<2pt>[l]^-{(0,1_Y)}}$$
In~\cite{Jacqmin2019}, they have been characterized as finitely complete pointed categories with $\Uni$-closed $\Mc_\strong$-relations. For the comparison, let us mention that unital categories can be described~\cite{Bourn1996} as finitely complete pointed categories such that, for each pair of objects $X,Y$, the morphisms $(1_X,0)$ and $(0,1_Y)$ are jointly strongly epimorphic.

In view of the above two examples, categories with $\M$-closed $\Mc_\strong$-relations were also called \emph{categories weakly with $\M$-closed relations} in~\cite{Jacqmin2019}.

To conclude this tour of examples, let us mention that, denoting the class of regular monomorphisms in a category by $\Mc_{\reg}$, finitely complete categories with $\Mal$-closed $\Mc_\reg$-relations are considered in~\cite{EJMF2023}. The class $\Mc_\reg$ is stable under pullbacks, but in general it fails to be closed under composition.

\subsection*{The algorithm for deciding implications of matrix properties}

As in~\cite{HJ2022,HJJ2022}, we call a \emph{matrix set} a subset of~$\matr_\ast$. A matrix set $S$ is said to be \emph{non-pointed} if $S\subseteq \matr$. Given a matrix set $S\subseteq\matr$ (respectively $S\subseteq\matr_\ast$), a finitely complete category~$\CC$ (respectively a finitely complete pointed category~$\CC$) and a class $\Mc$ of monomorphisms in $\CC$ stable under pullbacks, we say that $\CC$ \emph{has $S$-closed $\Mc$-relations} if $\CC$ has $\M$-closed $\Mc$-relations for all $\M$ in~$S$.

Extending the notation established in~\cite{HJ2023}, given a non-pointed matrix set $S\subseteq\matr$ and a matrix $\N\in\matr$, we write $S\Rightarrow_\lex^\all\N$, or simply $S\Rightarrow_\lex\N$, if any finitely complete category with $S$-closed relations has also $\N$-closed relations. The abbreviation `$\lex$' here stands for `left exact', another name for `finitely complete'. Similarly, we write $S\Rightarrow_\lex^\strong\N$ to mean that any finitely complete category with $S$-closed $\Mc_\strong$-relations has also $\N$-closed $\Mc_\strong$-relations. In the pointed case, i.e., given $S\subseteq\matr_\ast$ and $\N\in\matr_\ast$, we write $S\Rightarrow_{\lex_\ast}^\all\N$, or simply $S\Rightarrow_{\lex_\ast}\N$, to mean that any finitely complete pointed category with $S$-closed relations has also $\N$-closed relations. Finally, we write $S\Rightarrow_{\lex_\ast}^\strong\N$ if any finitely complete pointed category with $S$-closed $\Mc_\strong$-relations has also $\N$-closed $\Mc_\strong$-relations. In all these abbreviations, if $S=\{\M\}$ is a singleton, we sometimes write $\M$ instead of~$\{\M\}$, e.g., $\M\Rightarrow_\lex\N$ instead of $\{\M\}\Rightarrow_\lex\N$.

In~\cite{HJJ2022}, an algorithm was presented to decide whether $S\Rightarrow_\lex\N$ for a non-pointed matrix set $S\subseteq\matr$ and a matrix $\N\in\matr$. In~\cite{HJ2022}, this algorithm was adapted to the pointed context to decide whether $S\Rightarrow_{\lex_\ast}\N$ for a matrix set $S\subseteq\matr_\ast$ and a matrix $\N\in\matr_\ast$. Note that, in the original papers~\cite{HJ2022,HJJ2022}, it was required that $S$ is finite to be able to implement these algorithms on a computer. Since this will not be our concern in the present paper, we can omit this assumption here. Actually, from the results of~\cite{HJ2022, HJJ2022}, we can deduce the following (which can also be deduced from the algorithms recalled below since the number of columns that one can possibly add to $\N$ is finite).

\begin{proposition}
Given $S\subseteq\matr$ and $\N\in\matr$, one has $S\Rightarrow_\lex\N$ if and only if there is a finite subset $S'\subseteq S$ such that $S'\Rightarrow_\lex\N$. Analogously, given $S\subseteq\matr_\ast$ and $\N\in\matr_\ast$, one has $S\Rightarrow_{\lex_\ast}\N$ if and only if there is a finite subset $S'\subseteq S$ such that $S'\Rightarrow_{\lex_\ast}\N$.
\end{proposition}

As they will be needed, let us now recall these algorithms. In each case, the first step of the algorithm is to deal with trivial matrices. A matrix $\M\in\matr$ is said to be \emph{trivial}~\cite{HJJ2022} if every finitely complete category with $\M$-closed ($\Mc_\all$-)relations is a \emph{preorder}, i.e., for any two objects $X,Y$ in the category, there is at most one morphism $X\to Y$. Similarly, a matrix $\M\in\matr_\ast$ is said to be \emph{trivial}~\cite{HJ2022} if every finitely complete pointed category with $\M$-closed ($\Mc_\all$-)relations is a preorder (which in that case will contain only one isomorphism class of objects). Since $\matr\subset\matr_\ast$, it is sensible to ask whether these two definitions agree if $\M\in\matr$. As proved in~\cite{HJ2022}, they indeed do. In~\cite{HJ2022,HJJ2022}, it is shown how to efficiently recognize whether a matrix is trivial. As this will not be needed here, we do not enter into the details and refer the interested reader to those papers. However, some theoretical characterizations from these papers are recalled in Theorem~\ref{theorem trivial} where new characterizations are also proved. Given a matrix set $S\subseteq\matr_\ast$, we say that $S$ is \emph{trivial} if it contains a trivial matrix.

Now, given a non-pointed non-trivial matrix set $S\subseteq\matr$ and a matrix $\N\in\matr(n,m,k)$ for integers $n>0$ and $m,k\geqslant 0$, the algorithm from~\cite{HJJ2022} to decide whether $S\Rightarrow_\lex\N$ says the following:
\begin{quotation}
Keep expanding the set of left columns of~$\N$, until it is not possible anymore, with the right column of a row-wise interpretation $B$ of type $(\{x_1,\dots,x_k\},\dots,\{x_1,\dots,x_k\})$ of a matrix $\M '\in\matr(n,m',k')$ whose rows are rows of a common matrix $\M\in S$ such that the left columns of~$B$, but not its right column, can be found among the left columns of the expending~$\N$. One has the implication $S\Rightarrow_\lex\N$ if and only if the right column of $\N$ can be found among the left columns of this expanded~$\N$.
\end{quotation}

In the pointed case, given a non-trivial matrix set $S\subseteq \matr_\ast$ and a matrix $\N\in\matr_\ast(n,m,k)$ for integers $n>0$ and $m,k\geqslant 0$, the algorithm from~\cite{HJ2022} to decide whether $S\Rightarrow_{\lex_\ast}\N$ says the following (where the pointed set $\{\ast,x_1,\dots,x_k\}$ is again considered with $\ast$ as distinguished element):
\begin{quotation}
First, add to $\N$ a left column of $\ast$'s. Then, keep expanding the set of left columns of~$\N$, until it is no more possible, with the right column of a row-wise interpretation $B$ of type $(\{\ast,x_1,\dots,x_k\},\dots,\{\ast,x_1,\dots,x_k\})$ of a matrix $\M '\in\matr_\ast(n,m',k')$ whose rows are rows of a common matrix $\M\in S$ such that the left columns of~$B$, but not its right column, can be found among the left columns of the expanding~$\N$. One has the implication $S\Rightarrow_{\lex_\ast}\N$ if and only if the right column of $\N$ can be found among the left columns of this expanded~$\N$.
\end{quotation}

\subsection*{Categories of partial algebras}

In Universal Algebra, a \emph{(finitary one-sorted) signature} $\Sigma$ is a set of operation symbols $\sigma$, each of them being equipped with an integer $n\geqslant 0$ called its arity. Given such a signature $\Sigma$ and a set of variables~$X$, we define the set $\T_\Sigma(X)$ of \emph{$\Sigma$-terms in the variables from~$X$} inductively as follows:
\begin{itemize}
\item the elements of $X$ are terms (i.e., $X\subseteq \T_\Sigma(X)$);
\item given an $n$-ary operation symbol $\sigma\in\Sigma$ and $n$ terms $t_1,\dots,t_n\in\T_\Sigma(X)$, then $\sigma(t_1,\dots,t_n)$ is also a term in $\T_\Sigma(X)$. 
\end{itemize}
Every element of $\T_\Sigma(X)$ is constructed using only those two rules. A \emph{$\Sigma$-equation in the variables from $X$} is then a pair $(t,t')$ of terms in $\T_\Sigma(X)$. A \emph{$\Sigma$-equation} is then nothing else than a $\Sigma$-equation in the variables from some set~$X$.

For a signature $\Sigma$, a \emph{$\Sigma$-partial algebra} $A$ is a set (also denoted $A$ by abuse of notation) equipped with, for each $n$-ary operation symbol $\sigma\in\Sigma$, a \emph{partial} function $\sigma^A\colon A^n \dashrightarrow A$, i.e., a (total) function from a subset of $A^n$ to~$A$. In a $\Sigma$-partial algebra~$A$, each $\Sigma$-term $t$ in the variable from a set $X$ gives rise to a partial function $t^A\colon A^X\dashrightarrow A$ defined inductively as follows:
\begin{itemize}
\item if $t\in X$ is a variable, then $t^A$ is a total function given by the corresponding projection $A^X\to A\colon (a_x)_{x\in X} \mapsto a_t$;
\item if $t=\sigma(t_1,\dots,t_n)$ for some $n$-ary $\sigma\in\Sigma$ and $\Sigma$-terms $t_1,\dots,t_n\in\T_\Sigma(X)$, then $t^A$ is defined on the element $(a_x)_{x\in X}\in A^X$ if and only if $t^A_1,\dots,t^A_n$ are all defined on $(a_x)_{x\in X}$ and $\sigma^A$ is defined on $(t_1^A((a_x)_{x\in X}),\dots,t^A_n((a_x)_{x\in X}))$, in which case
$$t^A((a_x)_{x\in X})=\sigma^A(t_1^A((a_x)_{x\in X}),\dots,t^A_n((a_x)_{x\in X})).$$
\end{itemize}
Following the terminology from~\cite{Burmeister2004}, given a $\Sigma$-equation $(t,t')$ in the variables from~$X$, we say that $A$ satisfies the \emph{existence equation} $t=^\e t'$ (or simply that the equation $t=^\e t'$ holds in~$A$) if $t^A((a_x)_{x\in X})$ and $t'^A((a_x)_{x\in X})$ are both defined for any element $(a_x)_{x\in X}\in A^X$ and they are equal:
$$t^A((a_x)_{x\in X})=t'^A((a_x)_{x\in X}).$$
Although there are other kind of equations that can be defined for partial algebras, in this paper we will only consider existence equations.

Given a pair $(\Sigma,E)$ where $\Sigma$ is a signature and $E$ a set of existence $\Sigma$-equations, a \emph{$(\Sigma,E)$-partial algebra} is $\Sigma$-partial algebra which satisfies all equations of~$E$. A homomorphism $f\colon A\to B$ of $(\Sigma,E)$-partial algebras is a (total) function $f\colon A\to B$ such that, for each $n$-ary operation symbol $\sigma\in\Sigma$ and every $a_1,\dots,a_n\in A$, if $\sigma^A(a_1,\dots,a_n)$ is defined in~$A$, then $\sigma^B(f(a_1),\dots,f(a_n))$ is defined in $B$ and the equality
\begin{equation}\label{equation morphism partial algebras}
f(\sigma^A(a_1,\dots,a_n))=\sigma^B(f(a_1),\dots,f(a_n))
\end{equation}
holds. These form the category $\Part(\Sigma,E)$ of $(\Sigma,E)$-partial algebras. All small limits in such a category of partial algebras exist and are computed as in the category $\Set$ of sets and functions. More specifically, given a small category $\JJ$ and a diagram $D\colon\JJ\to\Part(\Sigma,E)$, the limit of $D$ is constructed as the set
$$L=\{(a_J\in D(J))_{J\in\JJ}\,|\, D(j)(a_{J})=a_{J'} \,\forall j\colon J\to J'\in\JJ\}$$
and, given any $n$-ary $\sigma\in\Sigma$,
$$\sigma^L((a_{1J})_{J\in\JJ},\dots,(a_{nJ})_{J\in\JJ})$$
is defined if and only if $\sigma^{D(J)}(a_{1J},\dots,a_{nJ})$ is defined for all object $J\in\JJ$, and is equal in that case to
$$\sigma^L((a_{1J})_{J\in\JJ},\dots,(a_{nJ})_{J\in\JJ})=(\sigma^{D(J)}(a_{1J},\dots,a_{nJ}))_{J\in\JJ}.$$
The legs of the limit are just the projections. Let us also mention that these categories of partial algebras are \emph{essentially algebraic}~\cite{AHR1988,AR1994}, i.e., \emph{locally presentable}~\cite{GU1971}. As we will need it later on, for a matrix set $S\subseteq\matr$ and a matrix $\N\in\matr$ (respectively a matrix set $S\subseteq\matr_\ast$ and a matrix $\N\in\matr_\ast$), we write $S\Rightarrow_{\essalg}^\strong\N$ (respectively $S\Rightarrow_{\essalg_\ast}^\strong\N$) to mean that any essentially algebraic category (respectively essentially algebraic pointed category) with $S$-closed $\Mc_\strong$-relations has also $\N$-closed $\Mc_\strong$-relations.

Inspired by Theorem~\ref{theorem characterization varieties with M-closed relations}, we now give our main example of a category of partial algebras. Given a matrix set $S\subseteq\matr$ (respectively $S\subseteq\matr_\ast$), we consider the signature $\Sigma^S$ which is constituted of, for each $\M\in S$ with $m$ left columns, an $m$-ary operation symbol $p^\M$ (respectively, the signature $\Sigma_\ast^S$ which is constituted of a nullary operation symbol~$0$ and, for each $\M\in S$ with $m$ left columns, an $m$-ary operation symbol $p^\M$). The set $E^S$ (respectively $E_\ast^S$) is constituted by one existence equation for each row of each $\M\in S$, namely
$$p^\M(x_{i1},\dots,x_{im})=^\e x_{i\,m+1}$$
in the variables from $\{x_1,\dots,x_k\}$ for the $i^{\textrm{th}}$ row of $\M$ if this row is
$$\left[\begin{array}{ccc|c} x_{i1} & \dots & x_{im} & x_{i\,m+1}\end{array}\right]$$
where the $x_{ij}$'s are in $\{x_1,\dots,x_k\}$ (respectively, in the pointed case, the $x_{ij}$'s are in $\{\ast,x_1,\dots,x_k\}$ and in this case $\ast$ is interpreted as $0$ in the above equation; we moreover add $0=^\e 0$ to $E_\ast^S$ to force $0$ to be defined). We denote the category of $(\Sigma^S,E^S)$-partial algebras (respectively of $(\Sigma_\ast^S,E_\ast^S)$-partial algebras) by $\Part^S$ (respectively by $\Part_\ast^S$).

In the pointed case $S\subseteq\matr_\ast$, for each $\M\in S$, since $\M$ has at least one row, this implies that the existence equation
$$p^\M(0,\dots,0)=^\e 0$$
holds in any $(\Sigma_\ast^S,E_\ast^S)$-partial algebra. This shows that, in the pointed case, the category $\Part_\ast^S$ is pointed.

Let us remark that for a non-pointed $S\subseteq\matr$ which contains a matrix $\M$ with no left columns, $p^\M$ is then a constant term and $p^\M=^\e x$ is an equation of~$E^S$. This equation means that, for any element $a$ of a $(\Sigma^S,E^S)$-partial algebra~$A$, $p^\M$ is defined and equal to~$a$. This does \emph{not} prevent $A$ to be empty. In that case, $\Part^S$ is thus a preorder with exactly two isomorphism classes of objects.

As it will useful later on, for a set~$X$, we denote by $\T^S(X)$ (respectively by $\T_\ast^S(X)$) the set of terms $\T_{\Sigma^S}(X)$ (respectively $\T_{\Sigma_\ast^S}(X)$). If $S=\{\M\}$ is a singleton, we make the usual abbreviations $\Sigma^\M:=\Sigma^{\{\M\}}$, $\Sigma_\ast^\M:=\Sigma_\ast^{\{\M\}}$, $E^\M:=E^{\{\M\}}$, $E_\ast^\M:=E_\ast^{\{\M\}}$, $\Part^\M:=\Part^{\{\M\}}$, $\Part_\ast^\M:=\Part_\ast^{\{\M\}}$, $\T^\M(X):=\T^{\{\M\}}(X)$ and $\T_\ast^\M(X):=\T_\ast^{\{\M\}}(X)$. The categories $\Part^\M$ and $\Part_\ast^\M$ have been introduced in~\cite{Jacqmin2019}. We conclude this subsection by mentioning an important property of $\Part^\M$ and $\Part_\ast^\M$, that we will generalize to the case of a matrix set in the next subsection.

\begin{proposition}\label{proposition PartM M-closed strong relations}\cite{Jacqmin2019}
Let $\M\in\matr$ (respectively $\M\in\matr_\ast$). The category $\Part^\M$ (respectively $\Part_\ast^\M$) has $\M$-closed $\Mc_\strong$-relations (in general, it does \emph{not} have $\M$-closed $\Mc_\all$-relations).
\end{proposition}

\subsection*{Free partial algebras}

Let us now fix two matrix sets $S'\subseteq S\subseteq\matr$ (respectively $S'\subseteq S\subseteq\matr_\ast$). We denote by $\U^{S,S'}\colon\Part^S\to\Part^{S'}$ (respectively $\U_\ast^{S,S'}\colon\Part_\ast^S\to\Part_\ast^{S'}$) the corresponding forgetful functor. In view of the description of small limits in categories of partial algebras recalled above, this functor preserves these small limits. Moreover, it is easy to construct its left adjoint. Given a $(\Sigma^{S'},E^{S'})$-partial algebra~$A$ (respectively a $(\Sigma_\ast^{S'},E_\ast^{S'})$-partial algebra~$A$), we distinguish two cases to construct the universal $(\Sigma^{S},E^{S})$-partial algebra~$\F^{S,S'}(A)$ (respectively the universal $(\Sigma_\ast^{S},E_\ast^{S})$-partial algebra~$\F_\ast^{S,S'}(A)$):
\begin{itemize}
\item In this first case, we suppose that $A$ has at least two elements but there is no $(\Sigma^{S},E^{S})$-partial algebra (respectively $(\Sigma_\ast^{S},E_\ast^{S})$-partial algebra) with at least two elements (i.e., $S$ is trivial anticipating the equivalences of Theorem~\ref{theorem trivial}). In that case, $\F^{S,S'}(A)$ (respectively $\F_\ast^{S,S'}(A)$) is the terminal object, that is, it has exactly one element and any $m$-ary operation symbol of $\Sigma^S$ (respectively of $\Sigma_\ast^S$) is defined on the unique element of $\F^{S,S'}(A)^m$ (respectively of $\F_\ast^{S,S'}(A)^m$). The reflection of $A$ along $\U^{S,S'}$ (respectively along $\U_\ast^{S,S'}$) is then given by the unique function $A\to \U^{S,S'}(\F^{S,S'}(A))$ (respectively $A\to \U_\ast^{S,S'}(\F_\ast^{S,S'}(A))$).
\item Otherwise, $\F^{S,S'}(A)$ (respectively $\F_\ast^{S,S'}(A)$) is obtained by keeping the same underlying $(\Sigma^{S'},E^{S'})$-partial algebra~$A$ (respectively the same underlying $(\Sigma_\ast^{S'},E_\ast^{S'})$-partial algebra~$A$) and defining, for each $\M\in {S\setminus S'}$, $p^\M$ only where it is required by the existence equations related to~$\M$. It is not hard to see that, in this case, these definitions do not contradict each other. The identity function $A\to \U^{S,S'}(\F^{S,S'}(A))$ (respectively $A\to\U_\ast^{S,S'}(\F_\ast^{S,S'}(A))$) gives the reflection of $A$ along $\U^{S,S'}$ (respectively $\U_\ast^{S,S'}$).
\end{itemize}
We have therefore an adjunction $\F^{S,S'}\dashv\U^{S,S'}$ (respectively $\F_\ast^{S,S'}\dashv\U_\ast^{S,S'}$). Moreover, we can see from the above construction of the left adjoint that, if $S$ is not trivial, this adjunction is actually a \emph{co-reflection}, i.e., the unit is an isomorphism. If $S'=\emptyset$ is empty, $\Part^{S'}$ is (equivalent to) the category $\Set$ of sets (respectively $\Part_\ast^{S'}$ is (equivalent to) the category $\Set_\ast$ of pointed sets). We therefore have a left adjoint $\F^S:=\F^{S,\emptyset}$ (respectively $\F_\ast^S:=\F_\ast^{S,\emptyset}$) for the forgetful functor $\U^S:=\U^{S,\emptyset}\colon\Part^S\to\Set$ (respectively $\U_\ast^S:=\U_\ast^{S,\emptyset}\colon\Part_\ast^S\to\Set_\ast$).

It is a classical result that right adjoints preserves strong monomorphisms (see e.g. Proposition~4.3.9 in~\cite{Borceux1994}). We therefore immediately have the following.

\begin{proposition}
Let $S'\subseteq S\subseteq\matr$ (respectively $S'\subseteq S\subseteq\matr_\ast$) be two matrix sets. The forgetful functor $\U^{S,S'}\colon\Part^S\to\Part^{S'}$ (respectively $\U_\ast^{S,S'}\colon\Part_\ast^S\to\Part_\ast^{S'}$) preserves strong monomorphisms.
\end{proposition}

Due to this proposition, we can extend Proposition~\ref{proposition PartM M-closed strong relations} to matrix sets.

\begin{proposition}\label{proposition PartS S-closed strong relations}
Let $S\subseteq\matr$ (respectively $S\subseteq\matr_\ast$) be a matrix set. The category $\Part^S$ (respectively $\Part^S_\ast$) has $S$-closed $\Mc_\strong$-relations.
\end{proposition}

\begin{proof}
Let $\M\in S\cap\matr_\ast(n,m,k)$ for integers $n>0$ and $m,k\geqslant 0$. Let also $r\colon R\rightarrowtail X^n$ be a $\Mc_\strong$-relation in $\Part^S$ (respectively in $\Part_\ast^S$) and
$$\left[\begin{array}{ccc|c} g_{11} & \dots & g_{1m} & h_1 \\ \vdots & & \vdots & \vdots \\ g_{n1} & \dots & g_{nm} & h_n \end{array}\right]$$
be an interpretation of $\M$ of type $\Part^S(Z,X)$ (respectively of type $\Part_\ast^S(Z,X)$) such that, for each $j\in\{1,\dots,m\}$, the morphism $(g_{1j},\dots,g_{nj})\colon Z\to X^n$ factors through~$r$. We must show that the morphism $(h_1,\dots,h_n)\colon Z\to X^n$ also factors through~$r$, i.e., that the pullback $r'$ of $r$ along $(h_1,\dots,h_n)$ is an isomorphism. Since the forgetful functor $\U^{S,\M}:=\U^{S,\{\M\}}$ (respectively $\U_\ast^{S,\M}:=\U_\ast^{S,\{\M\}}$) preserves limits and strong monomorphisms, and since $\Part^\M$ (respectively $\Part_\ast^\M$) has $\M$-closed $\Mc_\strong$-relations by Proposition~\ref{proposition PartM M-closed strong relations}, we know that $\U^{S,\M}(r')$ (respectively $\U_\ast^{S,\M}(r')$) is an isomorphism. This thus means that $r'$ is surjective and, being a $\Mc_\strong$-relation, it is an isomorphism.
\end{proof}

\subsection*{Closed sub-partial algebras}

Given a signature $\Sigma$ and a set $E$ of existence $\Sigma$-equations, the monomorphisms in $\Part(\Sigma,E)$ are exactly the injective homomorphisms. A \emph{sub-partial algebra} of a $(\Sigma,E)$-partial algebra $B$ is thus given by a $(\Sigma,E)$-partial algebra $A$ where the underlying set of $A$ is included in that of $B$ and such that, for any $n$-ary operation symbol $\sigma\in\Sigma$ and any $a_1,\dots,a_n\in A$, if $\sigma^A(a_1,\dots,a_n)$ is defined in~$A$, then $\sigma^B(a_1,\dots,a_n)$ is defined in $B$ and $\sigma^B(a_1,\dots,a_n)=\sigma^A(a_1,\dots,a_n)$. Note that in general, if $\sigma^B(a_1,\dots,a_n)$ is defined in~$B$, there is no reason for which $\sigma^A(a_1,\dots,a_n)$ is defined in~$A$. A sub-partial algebra for which this reverse implication holds is said to be \emph{closed}~\cite{Hoft1973}. More generally, a homomorphism $f\colon A\to B$ in $\Part(\Sigma,E)$ is said to be \emph{closed} if, for each $n$-ary operation symbol $\sigma\in\Sigma$ and every $a_1,\dots,a_n\in A$, if $\sigma^B(f(a_1),\dots,f(a_n))$ is defined in~$B$, then $\sigma^A(a_1,\dots,a_n)$ is defined in~$A$ (and therefore the identity~(\ref{equation morphism partial algebras}) holds).

Let us recall from~\cite{Jacqmin2019} that in most cases, closed monomorphisms in $\Part^\M$ and $\Part_\ast^\M$ coincide with strong monomorphisms.

\begin{proposition}\label{proposition closed and strong monos PartM}\cite{Jacqmin2019}
For a matrix $\M\in\matr_\ast$, closed monomorphisms in $\Part_\ast^\M$ coincide with strong monomorphisms. For a matrix $\M\in\matr$, strong monomorphisms in $\Part^\M$ are closed monomorphisms and the converse is true if there exists a partial algebra in $\Part^\M$ with at least two elements.
\end{proposition}

The condition that $\Part^\M$ contains a partial algebra with at least two elements in the above proposition will be rephrased as `$\M$ is not trivial', see the equivalence \ref{theorem trivial M trivial}$\Leftrightarrow$\ref{theorem trivial partial algebras max one element} in Theorem~\ref{theorem trivial}.

Using a similar same proof than in~\cite{Jacqmin2019}, we can generalize the above proposition to the case of matrix sets, giving rise to the following where we use Theorem~\ref{theorem trivial} by anticipation to reformulate the condition.

\begin{proposition}\label{proposition closed and strong monos PartS}
For a matrix set $S\subseteq\matr_\ast$, closed monomorphisms in $\Part_\ast^S$ coincide with strong monomorphisms. For a non-pointed matrix set $S\subseteq\matr$, strong monomorphisms in $\Part^S$ are closed monomorphisms and the converse is true if $S$ is not trivial.
\end{proposition}

\subsection*{The embedding theorem}

We now recall the embedding theorem from~\cite{Jacqmin2019} which will be of crucial importance in the proof of our results. The notation $\CC^{\op}$ stands for the dual of the category $\CC$ and, for a category~$\DD$, the notation $\DD^{\CC^{\op}}$ denotes the category of functors $\CC^{\op}\to\DD$ and natural transformations between them.

\begin{theorem}\label{theorem embedding theorem}\cite{Jacqmin2019}
Given a matrix $\M\in\matr$ (respectively $\M\in\matr_\ast$), a small finitely complete category~$\CC$ (respectively a small finitely complete pointed category~$\CC$) and $\Mc$ a class of monomorphisms in $\CC$ stable under pullbacks, closed under composition and containing the regular monomorphisms, if $\CC$ has $\M$-closed $\Mc$-relations, then there exists a fully faithful embedding $\varphi\colon\CC\hookrightarrow (\Part^\M)^{\CC^{\op}}$ (respectively $\varphi\colon\CC\hookrightarrow (\Part_\ast^\M)^{\CC^{\op}}$) which preserves and reflects finite limits and such that, for every monomorphism $f\in\Mc$ and every object $X\in\CC$, $\varphi(f)_X$ is a closed monomorphism in $\Part^\M$ (respectively in $\Part_\ast^\M$).
\end{theorem}

The embedding $\varphi$ can be constructed as a factorization of the Yoneda embedding $\Y\colon\CC\hookrightarrow\Set^{\CC^{\op}}$ (respectively $\Y_\ast\colon\CC\hookrightarrow\Set_\ast^{\CC^{\op}}$) through the functor $(\U^\M)^{\CC^{\op}}$ (respectively~$(\U_\ast^\M)^{\CC^{\op}}$) induced by post-composition with~$\U^\M:=\U^{\{\M\},\emptyset}$ (respectively with~$\U_\ast^\M:=\U_\ast^{\{\M\},\emptyset}$).
$$\xymatrix{ & (\Part^\M)^{\CC^{\op}} \ar[d]^-{(\U^\M)^{\CC^{\op}}} \\ \CC \ar@<2pt>@{^(->}[ru]^-{\varphi} \ar@<-2pt>@{^(->}[r]_-{\Y} & \Set^{\CC^{\op}}} \qquad\qquad \xymatrix{ & (\Part_\ast^\M)^{\CC^{\op}} \ar[d]^-{(\U_\ast^\M)^{\CC^{\op}}} \\ \CC \ar@<2pt>@{^(->}[ru]^-{\varphi} \ar@<-2pt>@{^(->}[r]_-{\Y_\ast} & \Set_\ast^{\CC^{\op}}}$$

Given a matrix set $S\subseteq\matr$ (respectively $S\subseteq\matr_\ast$), the category $\Part^S$ (respectively $\Part_\ast^S$) together with the forgetful functors $\U^S$ and $\U^{S,\M}$ for each $\M\in S$ (respectively $\U_\ast^S$ and $\U_\ast^{S,\M}$ for each $\M\in S$) has the universal property of the limit of the diagram constituted of one occurrence of $\Set$ (respectively of $\Set_\ast$) and, for each $\M\in S$, the forgetful functor $\U^\M$ (respectively $\U_\ast^\M$) to that occurrence of $\Set$ (respectively of $\Set_\ast$).
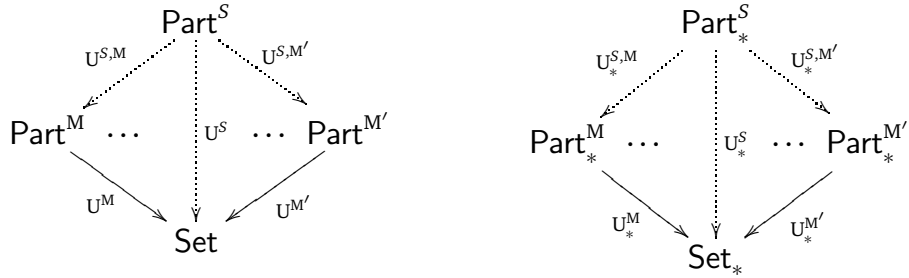
$$\xymatrix@C=1.5pt{&& \Part^S \ar@{.>}[lld]_-{\U^{S,\M}} \ar@{.>}[rrd]^-{\U^{S,\M'}} \ar@{.>}[dd]^-{\U^S} && \\ \Part^\M \ar[rrd]_-{\U^\M} & \cdots && \cdots & \Part^{\M'} \ar[lld]^-{\U^{\M'}} \\ && \Set &&} \qquad\qquad \xymatrix@C=1.5pt{&& \Part_\ast^S \ar@{.>}[lld]_-{\U_\ast^{S,\M}} \ar@{.>}[rrd]^-{\U_\ast^{S,\M'}} \ar@{.>}[dd]^-{\U_\ast^S} && \\ \Part_\ast^\M \ar[rrd]_-{\U_\ast^\M} & \cdots && \cdots & \Part_\ast^{\M'} \ar[lld]^-{\U_\ast^{\M'}} \\ && \Set_\ast &&}$$
In view of this, the above embedding theorem immediately generalizes for a matrix set as follows.

\begin{theorem}\label{theorem embedding theorem matrix set}
Given a matrix set $S\subseteq\matr$ (respectively $S\subseteq\matr_\ast$), a small finitely complete category~$\CC$ (respectively a small finitely complete pointed category~$\CC$) and $\Mc$ a class of monomorphisms in $\CC$ stable under pullbacks, closed under composition and containing the regular monomorphisms, if $\CC$ has $S$-closed $\Mc$-relations, then there exists a fully faithful embedding $\varphi\colon\CC\hookrightarrow (\Part^S)^{\CC^{\op}}$ (respectively $\varphi\colon\CC\hookrightarrow (\Part_\ast^S)^{\CC^{\op}}$) which preserves and reflects finite limits, such that, for every monomorphism $f\in\Mc$ and every object $X\in\CC$, $\varphi(f)_X$ is a closed monomorphism in $\Part^S$ (respectively in $\Part_\ast^S$) and which makes the diagram
$$\xymatrix{ & (\Part^S)^{\CC^{\op}} \ar[d]^-{(\U^S)^{\CC^{\op}}} \\ \CC \ar@<2pt>@{^(->}[ru]^-{\varphi} \ar@<-2pt>@{^(->}[r]_-{\Y} & \Set^{\CC^{\op}}}$$
respectively
$$\xymatrix{ & (\Part_\ast^S)^{\CC^{\op}} \ar[d]^-{(\U_\ast^S)^{\CC^{\op}}} \\ \CC \ar@<2pt>@{^(->}[ru]^-{\varphi} \ar@<-2pt>@{^(->}[r]_-{\Y_\ast} & \Set_\ast^{\CC^{\op}}}$$
commute.
\end{theorem}

\section{Main results}

We can now state and prove our results reinforcing the links between categories of partial algebras and matrix properties involving a general class $\Mc$ of monomorphisms.

\subsection*{Trivial matrix sets}

Let us start by studying trivial matrix sets. We extend here the characterizations of trivial matrices that can already be found in~\cite{HJ2022,HJ2023b,HJJ2022}. In addition of recalling some known characterizations, we add many new ones. We however omit to recall here many other characterizations found in those papers which allow to effectively recognize when a matrix is trivial. In the following theorem, $\Bool$ denotes the category of Boolean algebras and $\1$ denotes the category with a unique object and the identity on it as unique morphism. A category is said to be \emph{regular}~\cite{BGO1971} if it has finite limits, coequalizers of kernel pairs and such that regular epimorphisms are stable under pullbacks. Note that the following theorem can be used with $S=\{\M\}$ a singleton, giving rise to some characterizations of trivial matrices (instead of trivial matrix sets), some of those already appearing in~\cite{HJ2022,HJ2023b, HJJ2022}.

\begin{theorem}\label{theorem trivial}
For a matrix set $S\subseteq\matr_\ast$ , the following conditions are equivalent:
\begin{enumerate}[label=(\arabic*), ref=(\arabic*)]
\item\label{theorem trivial pointed S trivial} every finitely complete pointed category with $S$-closed relations is equivalent to the single morphism category~$\1$;
\item\label{theorem trivial pointed M trivial} $S$ is trivial, that is, $S$ contains a trivial matrix $\M$, i.e., such that every finitely complete pointed category with $\M$-closed relations is equivalent to the single morphism category~$\1$;
\item\label{theorem trivial pointed weakly S} every finitely complete pointed category with $S$-closed $\Mc_\strong$-relations is equivalent to the single morphism category~$\1$;
\item\label{theorem trivial pointed general class of monos} for every finitely complete pointed category $\CC$ and every class $\Mc$ of monomorphisms in $\CC$ stable under pullbacks, closed under composition, containing all regular monomorphisms and such that $\CC$ has $S$-closed $\Mc$-relations, $\CC$ is equivalent to the single morphism category~$\1$;
\item\label{theorem trivial pointed regular categories} any regular pointed category with $S$-closed relations is equivalent to the single morphism category~$\1$;
\item\label{theorem trivial pointed ess alg categories} any essentially algebraic pointed category with $S$-closed relations is equivalent to the single morphism category~$\1$;
\item\label{theorem trivial pointed varieties} any pointed variety with $S$-closed relations is equivalent to the single morphism category~$\1$, i.e., the equation $x = 0$ holds in its theory for a constant term~$0$;
\item\label{theorem trivial pointed PartS equivalent to 1} $\Part_\ast^S$ is equivalent to the single morphism category~$\1$;
\item\label{theorem trivial pointed partial algebra one element} every $(\Sigma_\ast^S,E_\ast^S)$-partial algebra has exactly one element;
\item\label{theorem trivial pointed Set op} $\Set_\ast^{\op}$ does \emph{not} have $S$-closed relations.
\end{enumerate}
If $S\subseteq\matr$ is non-pointed, these are further equivalent to the following conditions:
\begin{enumerate}[resume, label=(\arabic*), ref=(\arabic*)]
\item\label{theorem trivial S trivial} every finitely complete category with $S$-closed relations is a preorder;
\item\label{theorem trivial M trivial} $S$ contains a trivial matrix $\M$, i.e., such that every finitely complete category with $\M$-closed relations is a preorder;
\item\label{theorem trivial weakly S} every finitely complete category with $S$-closed $\Mc_\strong$-relations is a preorder;
\item\label{theorem trivial general class of monos} for every finitely complete category $\CC$ and every class $\Mc$ of monomorphisms in $\CC$ stable under pullbacks, closed under composition, containing all regular monomorphisms and such that $\CC$ has $S$-closed $\Mc$-relations, $\CC$ is a preorder;
\item\label{theorem trivial regular categories} any regular category with $S$-closed relations is a preorder;
\item\label{theorem trivial ess alg categories} any essentially algebraic category with $S$-closed relations is a preorder;
\item\label{theorem trivial varieties} any variety with $S$-closed relations is a preorder, i.e., the equation $x = y$ holds in its theory;
\item\label{theorem trivial PartS preorder} $\Part^S$ is a preorder;
\item\label{theorem trivial partial algebras max one element} every $(\Sigma^S,E^S)$-partial algebra has at most one element;
\item\label{theorem trivial Set op} $\Set^{\op}$ does \emph{not} have $S$-closed relations;
\item\label{theorem trivial Bool} $\Bool$ does \emph{not} have $S$-closed relations.
\end{enumerate}
\end{theorem}

\begin{proof}
We first consider the general case where $S\subseteq\matr_\ast$. The implication \ref{theorem trivial pointed M trivial}$\Rightarrow$\ref{theorem trivial pointed S trivial} is obvious while \ref{theorem trivial pointed S trivial}$\Rightarrow$\ref{theorem trivial pointed Set op} follows from the fact that $\Set_\ast^{\op}$ is not equivalent to~$\1$. If \ref{theorem trivial pointed Set op} holds, $S$ contains a matrix $\M$ for which $\Set_\ast^{\op}$ does not have $\M$-closed relations. Using Theorem~3.1 in~\cite{HJ2022}, this means that $\M$ is trivial. We have therefore already proved the equivalences \ref{theorem trivial pointed S trivial}$\Leftrightarrow$\ref{theorem trivial pointed M trivial}$\Leftrightarrow$\ref{theorem trivial pointed Set op}. The implication \ref{theorem trivial pointed general class of monos}$\Rightarrow$\ref{theorem trivial pointed weakly S} follows from the fact that the class of strong monomorphisms in a category is stable under pullbacks, closed under composition and contains regular monomorphisms. The implications \ref{theorem trivial pointed weakly S}$\Rightarrow$\ref{theorem trivial pointed S trivial}$\Rightarrow$\ref{theorem trivial pointed regular categories} are obvious, as well as \ref{theorem trivial pointed regular categories}$\Rightarrow$\ref{theorem trivial pointed varieties} since varieties of universal algebras form regular categories. Since essentially algebraic categories are complete and varieties are essentially algebraic categories, the implications \ref{theorem trivial pointed S trivial}$\Rightarrow$\ref{theorem trivial pointed ess alg categories}$\Rightarrow$\ref{theorem trivial pointed varieties} are also straightforward. Concerning the implication \ref{theorem trivial pointed varieties}$\Rightarrow$\ref{theorem trivial pointed partial algebra one element}, suppose $A$ is a $(\Sigma_\ast^S,E_\ast^S)$-partial algebra. Note that $A$ cannot be empty since $0$ needs to be defined in~$A$. For each matrix $\M\in S$ with $m$ left columns, the $m$-ary operation symbol $p^\M\in \Sigma_\ast^\M$ can be extended to a total operation $(p^\M)^A$ on $A$ by defining it to be $0$ on any $m$-tuple in $A^m$ on which it is not yet defined. Doing so, $A$ becomes a total $(\Sigma_\ast^S,E_\ast^S)$-algebra. Using Theorem~\ref{theorem characterization varieties with M-closed relations}, the variety of total $(\Sigma_\ast^S,E_\ast^S)$-algebras has $S$-closed relations. If \ref{theorem trivial pointed varieties} holds, this means that $A$ as a unique element, proving the implication \ref{theorem trivial pointed varieties}$\Rightarrow$\ref{theorem trivial pointed partial algebra one element}. To prove the equivalence \ref{theorem trivial pointed partial algebra one element}$\Leftrightarrow$\ref{theorem trivial pointed PartS equivalent to 1}, it suffices to notice that any $(\Sigma_\ast^S,E_\ast^S)$-partial algebra $A$ with a unique element is isomorphic to the terminal object of $\Part_\ast^S$ since $0$ has to be defined as well as $(p^\M)^A(0,\dots,0)$ for each $\M\in S$. We conclude the first part of the proof with the implication \ref{theorem trivial pointed partial algebra one element}$\Rightarrow$\ref{theorem trivial pointed general class of monos}. Let us suppose \ref{theorem trivial pointed partial algebra one element} holds and let $\CC$ be a finitely complete pointed category and $\Mc$ a class of monomorphisms in $\CC$ stable under pullbacks, closed under composition, containing all regular monomorphisms and such that $\CC$ has $S$-closed $\Mc$-relations. We shall prove that $\CC$ is equivalent to~$\1$. Since $\CC$ has a terminal object, it is not empty and thus it remains to prove that given any two objects $X,Y$ in~$\CC$, $0$ is the unique morphism $X\to Y$. If $\CC$ is a small category, one can use the Embedding Theorem~\ref{theorem embedding theorem matrix set} to obtain a fully faithful embedding $\varphi\colon\CC\hookrightarrow (\Part_\ast^S)^{\CC^{\op}}$. If there were two different morphisms $f,g\colon X\to Y$, then $\varphi(f)\neq\varphi(g)$ and so $\varphi(f)_Z\neq\varphi(g)_Z$ for some object~$Z\in\CC^{\op}$. However, this is not possible since $\Part_\ast^S$ is supposed to be equivalent to~$\1$. If $\CC$ is not small, one can extract from the proof of the Embedding Theorem~\ref{theorem embedding theorem} in~\cite{Jacqmin2019} a $(\Sigma_\ast^S,E_\ast^S)$-partial algebraic structure on $\CC(X,Y)$ and conclude by~\ref{theorem trivial pointed partial algebra one element} that it contains only the zero morphism.

Let us now suppose that $S\subseteq\matr$. The equivalence \ref{theorem trivial Set op}$\Leftrightarrow$\ref{theorem trivial Bool} can be found in~\cite{HJ2023b} while the equivalence \ref{theorem trivial pointed M trivial}$\Leftrightarrow$\ref{theorem trivial M trivial} is established in~\cite{HJ2022}. The remaining implications are proved similarly as above except that \ref{theorem trivial Set op}$\Rightarrow$\ref{theorem trivial M trivial} follows now from Theorem~2.3 in~\cite{HJJ2022} and \ref{theorem trivial PartS preorder}$\Rightarrow$\ref{theorem trivial partial algebras max one element} can be proved as follows. For any element $a$ in a $(\Sigma^S,E^S)$-partial algebra~$A$, one has that, for each $\M\in S$, $(p^\M)^A(a,\dots,a)$ is defined and equal to~$a$ (since $\M$ has at least one row). Such an element thus determines a morphism from the terminal $(\Sigma^S,E^S)$-partial algebra to~$A$, sending the unique element in the domain to~$a$. If $\Part^S$ is a preorder, this implies that $A$ cannot have more than one element.
\end{proof}

In~\cite{HJ2022,HJJ2022}, for a matrix set $S\subseteq\matr$ (respectively $S\subseteq\matr_\ast$), we denoted by $\mclex S$ (respectively by $\mclex_\ast S$) the collection of finitely complete (respectively finitely complete pointed) categories with $S$-closed relations. Similarly, we denote here by $\mclex^\strong S$ (respectively $\mclex_\ast^\strong S$) the collection of finitely complete (respectively finitely complete pointed) categories with $S$-closed $\Mc_\strong$-relations.

\begin{remark}\label{remark empty matrices}
Let $\M\in\matr(n,m,k)$ for integers $n>0$ and $m,k\geqslant 0$. If $m=0$, then $\M$ is automatically trivial and $\mclex\{\M\}$ is the collection of categories equivalent to~$\1$. If $\M$ is trivial but $m>0$, $\mclex\{\M\}$ is then the bigger collection of finitely complete preorders. Trivial matrix sets determine thus two different collections of categories when looking at $\Mc=\Mc_\all$. Although the equivalence \ref{theorem trivial S trivial}$\Leftrightarrow$\ref{theorem trivial weakly S} holds in the above theorem, trivial matrix sets determine only one collection of categories when looking at $\Mc=\Mc_\strong$. Indeed, if $S$ is trivial, $\mclex^\strong S$ is always the collection of finitely complete preorders. This is due to the fact that in finitely complete preorders, $\Mc_\strong$ is the class of isomorphisms and therefore they all have $S$-closed $\Mc_\strong$-relations even if $S$ contains a matrix with no left columns. This distinction disappears in the pointed case because if $S\subseteq\matr_\ast$ is a trivial matrix set, both $\mclex_\ast S$ and $\mclex_\ast^\strong S$ are the collection of categories equivalent to~$\1$.
\end{remark}

\subsection*{Implications of matrix properties}

We now come to our main result.

\begin{theorem}\label{theorem implication}
Let $S\subseteq\matr_\ast$ be a matrix set and $\N\in\matr_\ast$ be a matrix such that $\N$ has $m\geqslant 0$ left columns. The following conditions are equivalent:
\begin{enumerate}[label=(\arabic*), ref=(\arabic*)]
\item\label{theorem implication pointed lex} $S\Rightarrow_{\lex_\ast}\N$;
\item\label{theorem implication pointed strong} $S\Rightarrow_{\lex_\ast}^\strong\N$;
\item\label{theorem implication pointed essentially algebraic} $S\Rightarrow_{\essalg_\ast}^\strong\N$;
\item\label{theorem implication pointed general class of monos} for every finitely complete pointed category $\CC$ and every class $\Mc$ of monomorphisms in $\CC$ stable under pullbacks, closed under composition and containing all regular monomorphisms, if $\CC$ has $S$-closed $\Mc$-relations, then it has also $\N$-closed $\Mc$-relations;
\item\label{theorem implication pointed PartS} $\Part_\ast^S$ has $\N$-closed $\Mc_\strong$-relations;
\item\label{theorem implication pointed term} there exists an $m$-ary term $q\in\T_\ast^S(\{y_1,\dots,y_m\})$ such that, for any $(\Sigma_\ast^S,E_\ast^S)$-partial algebra~$A$, $q^A\colon A^m\dashrightarrow A$ makes $A$ a $(\Sigma_\ast^\N,E_\ast^\N)$-partial algebra.
\end{enumerate}
Moreover, if $S\subseteq\matr$, $\N\in\matr$ and $m>0$, these are further equivalent to the following conditions:
\begin{enumerate}[resume, label=(\arabic*), ref=(\arabic*)]
\item\label{theorem implication lex} $S\Rightarrow_{\lex}\N$;
\item\label{theorem implication strong} $S\Rightarrow_{\lex}^\strong\N$;
\item\label{theorem implication essentially algebraic} $S\Rightarrow_{\essalg}^\strong\N$;
\item\label{theorem implication general class of monos} for every finitely complete category $\CC$ and every class $\Mc$ of monomorphisms in $\CC$ stable under pullbacks, closed under composition and containing all regular monomorphisms, if $\CC$ has $S$-closed $\Mc$-relations, then it has also $\N$-closed $\Mc$-relations;
\item\label{theorem implication PartM} $\Part^S$ has $\N$-closed $\Mc_\strong$-relations;
\item\label{theorem implication term} there exists an $m$-ary term $q\in\T^S(\{y_1,\dots,y_m\})$ such that, for any $(\Sigma^S,E^S)$-partial algebra~$A$, $q^A\colon A^m\dashrightarrow A$ makes $A$ a $(\Sigma^\N,E^\N)$-partial algebra.
\end{enumerate}
\end{theorem}

\begin{proof}
We first consider the general case where $S\subseteq\matr_\ast$ and $\N\in\matr_\ast$. To fix dimensions, we suppose that $\N\in\matr_\ast(n,m,k)$ for integers $n>0$ and $m,k\geqslant 0$. The implications \ref{theorem implication pointed general class of monos}$\Rightarrow$\ref{theorem implication pointed lex} and \ref{theorem implication pointed general class of monos}$\Rightarrow$\ref{theorem implication pointed strong} are obvious since the classes $\Mc_\all$ and $\Mc_\strong$ in any finitely complete category are closed under composition, stable under pullbacks and contain all regular monomorphisms. The implication \ref{theorem implication pointed strong}$\Rightarrow$\ref{theorem implication pointed essentially algebraic} is trivial since essentially algebraic categories are complete. The implication \ref{theorem implication pointed essentially algebraic}$\Rightarrow$\ref{theorem implication pointed PartS} follows immediately from Proposition~\ref{proposition PartS S-closed strong relations} and the fact that $\Part_\ast^S$ is an essentially algebraic pointed category. To prove the implication \ref{theorem implication pointed PartS}$\Rightarrow$\ref{theorem implication pointed general class of monos}, let us suppose that $\Part_\ast^S$ has $\N$-closed $\Mc_\strong$-relations and let us consider a finitely complete pointed category $\CC$ and a class of monomorphisms $\Mc$ in $\CC$ as in~\ref{theorem implication pointed general class of monos}. We suppose that $\CC$ has $S$-closed $\Mc$-relations and we must show it has $\N$-closed $\Mc$-relations. Let $r\colon R\rightarrowtail X^n$ be a relation in~$\Mc$, $Z$ be an object of $\CC$ and
$$\left[\begin{array}{ccc|c} g_{11} & \dots & g_{1m} & h_1 \\ \vdots & & \vdots & \vdots \\ g_{n1} & \dots & g_{nm} & h_{n} \end{array}\right]$$
be an interpretation of $\N$ of type $\CC(Z,X)$. We suppose that, for each $j\in\{1,\dots,m\}$, the morphism $(g_{1j},\dots,g_{nj})\colon Z\to X^{n}$ factors through $r$ and we must show that $(h_1,\dots,h_{n})\colon Z\to X^{n}$ also factors through~$r$. If $\CC$ is a small category, by the Embedding Theorem~\ref{theorem embedding theorem matrix set}, there exists a fully faithful embedding $\varphi\colon \CC\hookrightarrow (\Part_\ast^S)^{\CC^{\op}}$ which preserves and reflects finite limits and such that, for every monomorphism $f\in\Mc$ and every object $A\in\CC$, $\varphi(f)_A$ is a closed monomorphism in $\Part_\ast^S$. By Proposition~\ref{proposition closed and strong monos PartS}, $\varphi(r)_A$ is a strong relation for every object $A\in\CC$. Since $\Part_\ast^S$ is supposed to have $\N$-closed $\Mc_\strong$-relations, this means that $\varphi((h_1,\dots,h_{n}))_A$ factors through $\varphi(r)_A$ for every $A\in\CC$. Therefore, the pullback of $\varphi(r)$ along $\varphi((h_1,\dots,h_{n}))$ is an isomorphism since each of its component is. Thus, since $\varphi$ is fully faithful and preserves finite limits, the pullback of $r$ along $(h_1,\dots,h_{n})$ is $\CC$ is an isomorphism, proving that $(h_1,\dots,h_{n})$ factors through~$r$. If $\CC$ is not small, a similar argument can still be applied. Indeed, looking at the proof in~\cite{Jacqmin2019} of Theorem~\ref{theorem embedding theorem}, we can still construct $(\Sigma_\ast^S,E_\ast^S)$-partial algebraic structures on $\CC(Z,X)$, $\CC(Z,X^n)$ and $\CC(Z,R)$. This can be done in such a way that $\CC(Z,X^n)$ is isomorphic as $(\Sigma_\ast^S,E_\ast^S)$-partial algebra to the $n$-fold power of $\CC(Z,X)$ via the canonical isomorphism $\CC(Z,X^n)\cong\CC(Z,X)^n$ and such that composition with $r$ gives a closed monomorphism $\CC(Z,R)\rightarrowtail\CC(Z,X^n)$. We can thus see $\CC(Z,R)$ as a closed sub-partial algebra of $\CC(Z,X)^n$ containing $(g_1,\dots,g_{n})\in\CC(Z,X)^n$ if and only if the morphism $(g_1,\dots,g_{n})\colon Z\to X^n$ factors through~$r$. By Proposition~\ref{proposition closed and strong monos PartS}, this thus defines a strong $n$-ary relation in $\Part_\ast^S$. Therefore this relation is $\N$-closed by assumption. Considering the free $(\Sigma_\ast^S,E_\ast^S)$-partial algebra $\F_\ast^S(\{\ast,x\})$ on the one element set (notice that $\{\ast,x\}$ is the free pointed set on the one element set~$\{x\}$), we get an interpretation
$$\left[\begin{array}{ccc|c} g'_{11} & \dots & g'_{1m} & h'_1 \\ \vdots & & \vdots & \vdots \\ g'_{n1} & \dots & g'_{nm} & h'_{n} \end{array}\right]$$
of $\N$ of type $\Part_\ast^S(\F_\ast^S(\{\ast,x\}),\CC(Z,X))$ where $g'_{ij}(x)=g_{ij}$ and $h'_i(x)=h_i$ for all $i\in\{1,\dots,n\}$ and $j\in\{1,\dots,m\}$. Since $(g_{1j},\dots,g_{nj})$ factors through $r$ for each $j\in\{1,\dots,m\}$, we know that the morphisms $(g'_{1j},\dots,g'_{nj})\colon\F_\ast^S(\{\ast,x\})\to\CC(Z,X)^n$ factor through $\CC(Z,R)$. We conclude from this that $(h'_1,\dots,h'_{n})$ also factors through $\CC(Z,R)$ meaning that $(h_1,\dots,h_{n})$ factors through $r$ as required. We have therefore already proved the implications \ref{theorem implication pointed strong}$\Leftrightarrow$\ref{theorem implication pointed essentially algebraic}$\Leftrightarrow$\ref{theorem implication pointed general class of monos}$\Leftrightarrow$\ref{theorem implication pointed PartS}$\Rightarrow$\ref{theorem implication pointed lex}.

Let us now prove \ref{theorem implication pointed lex}$\Rightarrow$\ref{theorem implication pointed term}. If $S$ is trivial, by the equivalence \ref{theorem trivial pointed M trivial}$\Leftrightarrow$\ref{theorem trivial pointed partial algebra one element} of Theorem~\ref{theorem trivial}, it suffices to choose $q$ to be the constant term~$0$. If $S$ is not trivial, we rely on the algorithm to decide $S\Rightarrow_{\lex_\ast}\N$ from~\cite{HJ2022} recalled in Section~\ref{section preliminaries}. We denote the entries of the matrix $\N$ as in
$$\N=\left[\begin{array}{ccc|c} x_{11} & \dots & x_{1m} & x_{1\,m+1} \\ \vdots & & \vdots & \vdots \\ x_{n1} & \dots & x_{nm} & x_{n\,m+1} \end{array}\right]$$
where the $x_{ij}$'s belong to $\{\ast,x_1,\dots,x_{k}\}$. To each left column $c$ of~$\N$ (the original ones and the ones added by the algorithm), we associate an $m$-ary term $q_c\in\T_\ast^S(\{y_1,\dots,y_{m}\})$ such that, for each $i\in\{1,\dots,n\}$, the equation
$$q_c(x_{i1},\dots,x_{im})=^\e c_i$$
holds in any $(\Sigma_\ast^S,E_\ast^S)$-partial algebra, where $c_i$ is the $i^{\textrm{th}}$ entry of $c$ and $\ast$ is interpreted as~$0$. For the $j^{\textrm{th}}$ original left column $c^j$ of~$\N$ (for $j\in\{1,\dots,m\}$), we simply let $q_{c^j}$ be~$y_j$. For the first added column $c^{j+1}$ of~$\ast$'s, we take $q_{c^{j+1}}$ to be the term~$0$. Now, suppose at some point of the algorithm, we have defined $q_c$ for all the left columns of the current $\N$ and we are adding the right column $c^{B,m'+1}$ of a row-wise interpretation $B$ of type $(\{\ast,x_1,\dots,x_{k}\},\dots,\{\ast,x_1,\dots,x_{k}\})$ of a matrix $\M '\in\matr_\ast(n,m',k')$ whose rows are rows of a common matrix $\M\in S\cap\matr_\ast(n',m',k')$ (for integers $n'>0$, $m',k'\geqslant 0$) and such that the left columns $c^{B,j}$ of~$B$ (for each $j\in\{1,\dots,m'\}$), but not its right column $c^{B,m'+1}$, can be found among the left columns of~$\N$. We then define
$$q_{c^{B,m'+1}}=p^\M(q_{c^{B,1}},\dots,q_{c^{B,m'}}).$$
For $i\in\{1,\dots,n\}$, the $i^{\textrm{th}}$ row of $B$ is of the form
$$\left[\begin{array}{ccc|c} f(x'_{i'1}) & \dots & f(x'_{i'm'}) & f(x'_{i'\,m'+1}) \end{array}\right]$$
for a pointed function $f\colon\{\ast,x_1,\dots,x_{k'}\}\to\{\ast,x_1,\dots,x_{k}\}$ and an $i'\in\{1,\dots,n'\}$ and where
$$\left[\begin{array}{ccc|c} x'_{i'1} & \dots & x'_{i'm'} & x'_{i'\,m'+1} \end{array}\right]$$
is the $i'^{\textrm{th}}$ row of~$\M$. Thus, in any $(\Sigma_\ast^S,E_\ast^S)$-partial algebra, the equations
\begin{align*}
c^{B,m'+1}_{i} &=^\e f(x'_{i'\,m'+1})\\
&=^\e p^\M(f(x'_{i'1}),\dots,f(x'_{i'm'}))\\
&=^\e p^\M(c^{B,1}_{i},\dots,c^{B,m'}_{i})\\
&=^\e p^\M(q_{c^{B,1}}(x_{i1},\dots,x_{im}),\dots,q_{c^{B,m'}}(x_{i1},\dots,x_{im}))\\
&=^\e q_{c^{B,m'+1}}(x_{i1},\dots,x_{im})
\end{align*}
hold, completing the construction of the~$q_c$'s. But this immediately proves \ref{theorem implication pointed lex}$\Rightarrow$\ref{theorem implication pointed term} since, if $S\Rightarrow_{\lex_\ast}\N$ holds, the algorithm says that the right column of $\N$ has been added by the algorithm to its left columns.

To conclude the first list of equivalences, it remains to show \ref{theorem implication pointed term}$\Rightarrow$\ref{theorem implication pointed PartS}. Supposing~\ref{theorem implication pointed term}, we get a forgetful functor $U\colon\Part_\ast^S\to\Part_\ast^\N$. In view of the construction of small limits in categories of partial algebras, it is routine to show that $U$ preserves them. In order to prove~\ref{theorem implication pointed PartS}, let us consider a strong monomorphism $r\colon R\rightarrowtail X^{n}$ in $\Part_\ast^S$, a $(\Sigma_\ast^S,E_\ast^S)$-partial algebra $Z$ and an interpretation
$$\left[\begin{array}{ccc|c} g_{11} & \dots & g_{1m} & h_1 \\ \vdots & & \vdots & \vdots \\ g_{n1} & \dots & g_{nm} & h_{n} \end{array}\right]$$
of $\N$ of type $\Part_\ast^S(Z,X)$. Let us suppose that for each $j\in\{1,\dots,m\}$, the morphism $(g_{1j},\dots,g_{nj})\colon Z\to X^{n}$ factors through~$r$. Since by Proposition~\ref{proposition PartM M-closed strong relations}, $\Part_\ast^\N$ has $\N$-closed $\Mc_\strong$-relations, applying the functor~$U$, we get that
$$(U(h_{1}),\dots,U(h_{n}))\colon U(Z)\to U(X)^{n}$$
(which is (isomorphic to) $U((h_{1},\dots,h_{n}))\colon U(Z)\to U(X^{n})$) factors through~$U(r)$. Therefore, the pullback $r'$ of $r$ along $(h_{1},\dots,h_{n})$ is sent by $U$ to an isomorphism, and so it is bijective. So $r'$ is an epimorphism which, being a pullback of the strong monomorphism~$r$, is also a strong monomorphism. It is thus an isomorphism, which shows that $(h_{1},\dots,h_{n})$ factors through $r$ as desired.

Let us now consider the case where $S\subseteq\matr$, $\N\in\matr$ and $m>0$. The equivalence \ref{theorem implication pointed lex}$\Leftrightarrow$\ref{theorem implication lex} has been proved in~\cite{HJ2022} using the assumption $m>0$. The other equivalences can be proved analogously as in the pointed case above. For the proof of \ref{theorem implication PartM}$\Rightarrow$\ref{theorem implication general class of monos}, we separate it in two cases. It is analogous to the one of \ref{theorem implication pointed PartS}$\Rightarrow$\ref{theorem implication pointed general class of monos} if $S$ is not trivial (to be able to apply Proposition~\ref{proposition closed and strong monos PartS}). If $S$ is trivial, \ref{theorem implication general class of monos} is always true since $m>0$ (see Theorem~\ref{theorem trivial} and Remark~\ref{remark empty matrices}). For the implication \ref{theorem implication lex}$\Rightarrow$\ref{theorem implication term}, it is analogous to \ref{theorem implication pointed lex}$\Rightarrow$\ref{theorem implication pointed term} if $S$ is not trivial. If $S$ is trivial, since $m>0$ and in view of the equivalence \ref{theorem trivial pointed M trivial}$\Leftrightarrow$\ref{theorem trivial partial algebras max one element} of Theorem~\ref{theorem trivial}, one can choose $q$ to be the term~$y_1$.
\end{proof}

\begin{remark}\label{remark m'=0 in theorem implication}
Let us now say a word on the case where $S\subseteq\matr$, $\N\in\matr$ and $m=0$ in the above theorem. Using Theorem~\ref{theorem trivial} and Remark~\ref{remark empty matrices}, the statements \ref{theorem implication pointed lex}, \ref{theorem implication pointed strong}, \ref{theorem implication pointed essentially algebraic}, \ref{theorem implication pointed general class of monos}, \ref{theorem implication pointed PartS}, \ref{theorem implication pointed term}, \ref{theorem implication strong}, \ref{theorem implication essentially algebraic} and \ref{theorem implication PartM} of Theorem~\ref{theorem implication} are in that case all equivalent to the condition that $S$ is trivial. Besides, the statements \ref{theorem implication lex}, \ref{theorem implication general class of monos} and \ref{theorem implication term} of Theorem~\ref{theorem implication} are in that case equivalent to the condition that $S$ contains a matrix with no left columns.
\end{remark}

\begin{openquestion}
We leave open the question whether the statements of Theorem~\ref{theorem implication} are further equivalent to $S\Rightarrow_{\essalg_\ast}\N$ (in the general case) and to $S\Rightarrow_{\essalg}\N$ (in the case $S\subseteq\matr$, $\N\in\matr$ and $m>0$) with the obvious meaning of these notations. The fact that we have been unable yet to prove or disprove this reduction is closely linked with the fact that we do not know yet if the property of being Mal'tsev (or having $\M$-closed relations for a general matrix~$\M$) is stable under the pro-completion of a finitely complete (pointed) category, see~\cite{JJ2021}.
\end{openquestion}

As a consequence of Theorem~\ref{theorem implication}, one can now write proofs of implications of matrix properties in another way. In~\cite{HJ2022,HJJ2022}, lex-tableaux and $\text{lex}_\ast$-tableaux were used to represent such proofs. For instance, if $\Ari$ and $\Maj$ are the arithmetical and majority matrices from Section~\ref{section preliminaries}, a proof of the implication $\Ari\Rightarrow_{\lex}\Maj$ can been displayed as the following lex-tableau:
$$\begin{array}{ccc|c||ccc|c} 
x_1 & x_2 & x_2 & x_1 & x_1 & x_1 & x_2 & x_1 \\
x_2 & x_2 & x_1 & x_1 & x_1 & x_2 & x_1 & x_1 \\
x_1 & x_2 & x_1 & x_1 & x_2 & x_1 & x_1 & x_1 \\
\hline
x_2 & x_1 & x_1 & x_2 & x_2 &     &     &     \\
x_1 & x_1 & x_2 & x_2 & x_2 &     &     &     \\
x_1 & x_2 & x_1 & x_1 & x_1 &     &     &     \\
\hline
x_1 & x_2 & x_2 & x_1 & x_1 &     &     &     \\
x_2 & x_2 & x_1 & x_1 & x_1 &     &     &     \\
x_1 & x_1 & x_1 & x_1 & x_1 &     &     &     \\
\end{array}$$
The first block on the left represents the matrix $\Ari$, while the first block on the right represents the matrix $\Maj$. Below each horizontal line, a step of the algorithm to decide $\Ari\Rightarrow_{\lex}\Maj$ is displayed. So, for the second block, a row-wise interpretation of a matrix whose rows are rows of $\Ari$ (the matrix $\Ari$ itself in this case) is drawn, where each left column belongs to the left columns of $\Maj$ but not its right column. On the right, this right column is added to the extended $\Maj$. In the third block, again a row-wise interpretation of $\Ari$ is displayed on the left and its right column is added to the extended $\Maj$ on the right. Since this is the right column of $\Maj$, this concludes the proof of $\Ari\Rightarrow_{\lex}\Maj$.

One can transform this proof into a ternary term $q\in\T^\Ari(\{y_1,y_2,y_3\})$ proving the implication $\Ari\Rightarrow_{\lex}\Maj$ via the equivalence \ref{theorem implication lex}$\Leftrightarrow$\ref{theorem implication term} of Theorem~\ref{theorem implication}. To extract this term $q$ from the above lex-tableau, one follows the proof of \ref{theorem implication lex}$\Rightarrow$\ref{theorem implication term}, analogous to the proof of \ref{theorem implication pointed lex}$\Rightarrow$\ref{theorem implication pointed term}. To the first three left columns of $\Maj$, one associates the variables $y_1,y_2,y_3$ respectively. To the first added column
$$\left[\begin{array}{c} x_2 \\ x_2 \\ x_1 \end{array}\right]$$
one associates the term $p(y_3,y_1,y_2)$ where $p=p^\Ari$ is the term from $\Ari$, i.e., partial algebras in $\Part^\Ari$ satisfy the equations
\begin{equation}\label{equation pixley}
\begin{cases}p(x_1,x_2,x_2)=^\e x_1\\ p(x_2,x_2,x_1)=^\e x_1\\p(x_1,x_2,x_1)=^\e x_1.\end{cases}
\end{equation}
Finally, the term $q$, corresponding to the second added column is
$$q(y_1,y_2,y_3)=p(y_2,p(y_3,y_1,y_2),y_3).$$
At each step, the term is obtained applying $p$ to the terms corresponding to the left columns of the row-wise interpretation. To prove $\Ari\Rightarrow_{\lex}\Maj$ it is actually enough to check that, from the equations~(\ref{equation pixley}), one has the equations
$$q(x_1,x_1,x_2)=^\e p(x_1,p(x_2,x_1,x_1),x_2) =^\e p(x_1,x_2,x_2) =^\e x_1,$$
$$q(x_1,x_2,x_1)=^\e p(x_2,p(x_1,x_1,x_2),x_1) =^\e p(x_2,x_2,x_1) =^\e x_1$$
and
$$q(x_2,x_1,x_1)=^\e p(x_1,p(x_1,x_2,x_1),x_1) =^\e p(x_1,x_1,x_1) =^\e x_1$$
corresponding to the matrix $\Maj$. 

Since both rows of $\Mal$ appear in $\Ari$, it is straightforward to prove $\Ari\Rightarrow_\lex\Mal$ using the above term proofs. As shown in~\cite{HJJ2022}, one also has $\{\Maj,\Mal\}\Rightarrow_\lex\Ari$. In the language of terms, this can be done by considering
$$q(x,y,z)=p^\Mal(x,p^\Maj(z,x,y),z)$$
where $p^\Mal$ and $p^\Maj$ are terms satisfying the existence equations from $\Mal$ and $\Maj$ respectively. This term $q$ indeed satisfies the existence equations from $\Ari$ since
$$q(x,y,y)=^\e p^\Mal(x,p^\Maj(y,x,y),y)=^\e p^\Mal(x,y,y)=^\e x$$
$$q(y,y,x)=^\e p^\Mal(y,p^\Maj(x,y,y),x)=^\e p^\Mal(y,y,x)=^\e x$$
$$q(x,y,x)=^\e p^\Mal(x,p^\Maj(x,x,y),x)=^\e p^\Mal(x,x,x)=^\e x.$$

We have thus re-proved that an arithmetical category is the same thing as a Mal'tsev majority category. By the equivalence \ref{theorem implication strong}$\Leftrightarrow$\ref{theorem implication term} of Theorem~\ref{theorem implication}, this also shows that a finitely complete category is \emph{weakly arithmetical} (i.e., has $\Ari$-closed $\Mc_\strong$-relations) if and only if it is \emph{weakly Mal'tsev} (i.e., has $\Mal$-closed $\Mc_\strong$-relations) and \emph{weakly majority} (i.e., has $\Maj$-closed $\Mc_\strong$-relations).

To give an example in the pointed case, let us mention that in~\cite{HJ2022} we have shown that the matrix $\StrUni$ is equivalent to the matrix
$$\StrUni'=\left[\begin{array}{ccc|c}
x_1  & x_1  & \ast & x_1 \\
\ast & \ast & x_1  & x_1 \\
x_1  & \ast & x_1  & \ast
\end{array}\right]$$
in the sense that $\StrUni\Rightarrow_{\lex_\ast}\StrUni'$ and $\StrUni'\Rightarrow_{\lex_\ast}\StrUni$. Let us re-prove it here using (partial) terms. The term $p(x,y,z)$ corresponding to $\StrUni$ is associated to the equations
\begin{equation}\label{equation StrUni}
\begin{cases}p(x,0,0)=^\e x\\ p(y,y,x)=^\e x\end{cases}
\end{equation}
while the term $q(x,y,z)$ corresponding to $\StrUni'$ is associated to the equations
\begin{equation}\label{equation StrUni 2}
\begin{cases}q(x,x,0)=^\e x\\ q(0,0,x)=^\e x \\ q(x,0,x)=^\e 0.\end{cases}
\end{equation}
From a term $p(x,y,z)$ satisfying the equations~(\ref{equation StrUni}), we can construct the term
$$q(x,y,z)=p(z,p(x,y,0),y)$$
satisfying the equations~(\ref{equation StrUni 2}). Indeed,
$$\begin{cases}q(x,x,0)=^\e p(0,p(x,x,0),x)=^\e p(0,0,x)=^\e x\\
q(0,0,x) =^\e p(x,p(0,0,0),0)=^\e p(x,0,0)=^\e x\\
q(x,0,x)=^\e p(x,p(x,0,0),0)=^\e p(x,x,0)=^\e 0\end{cases}$$
already proving $\StrUni\Rightarrow_{\lex_\ast}\StrUni'$. For the other implication, from a term $q(x,y,z)$ satisfying the equations~(\ref{equation StrUni 2}), we can construct the term
$$p(x,y,z)=q(q(y,0,x),q(y,0,x),z)$$
satisfying the equations~(\ref{equation StrUni}). Indeed,
$$\begin{cases}p(x,0,0)=^\e q(q(0,0,x),q(0,0,x),0)=^\e q(x,x,0)=^\e x\\
p(y,y,x)=^\e q(q(y,0,y),q(y,0,y),x)=^\e q(0,0,x)=^\e x\end{cases}$$
proving $\StrUni'\Rightarrow_{\lex_\ast}\StrUni$.

\begin{remark}
Let us make clear that in the above proofs, \emph{partial} terms are considered. It was shown in~\cite{HJJ2022} that matrix implications are context sensitive. In particular, given matrices $\M,\N\in\matr$, if one proves that any variety with $\M$-closed relations has $\N$-closed relations, this does \emph{not} imply that $\M\Rightarrow_\lex\N$ in general. An example of this phenomenon has already been shown in the Introduction. Let us show here an example in the non-pointed context. We consider the matrices
$$\Cube^{\Delta^*}_3=\left[\begin{array}{ccccc|c}
x_1 & x_1 & x_1 & x_2 & x_2 & x_1 \\
x_1 & x_2 & x_2 & x_1 & x_1 & x_1 \\
x_2 & x_1 & x_2 & x_1 & x_2 & x_1
\end{array}\right]$$
and
$$\Edge_3=\left[\begin{array}{cccc|c}
x_2 & x_2 & x_1 & x_1 & x_1 \\
x_2 & x_1 & x_2 & x_1 & x_1 \\
x_1 & x_1 & x_1 & x_2 & x_1
\end{array}\right].$$
It has been shown in~\cite{HJJ2022} (using a computer implementation of the algorithm, see Figure~2 in~\cite{HJJ2022}) that $\Edge_3\Rightarrow_\lex\Cube^{\Delta^*}_3$ but $\Cube^{\Delta^*}_3\nRightarrow_\lex\Edge_3$ (i.e., the implication $\Cube^{\Delta^*}_3\Rightarrow_\lex\Edge_3$ does \emph{not} hold). However, it was shown in~\cite{BIMMVW2010} that varieties with $\Cube^{\Delta^*}_3$-closed relations also have $\Edge_3$-closed relations. According to Theorem~\ref{theorem characterization varieties with M-closed relations}, a variety $\VV$ has $\Cube^{\Delta^*}_3$-closed relations if and only if its theory has a $5$-ary term $p$ satisfying the equations
$$\begin{cases}p(x_1,x_1,x_1,x_2,x_2)=x_1\\
p(x_1,x_2,x_2,x_1,x_1)=x_1\\
p(x_2,x_1,x_2,x_1,x_2)=x_1.\end{cases}$$
Such a term has been called a \emph{$\Delta^*$-special cube term} in~\cite{BIMMVW2010}, explaining the notation $\Cube^{\Delta^*}_3$. Analogously, a variety $\VV$ has $\Edge_3$-closed relations if and only if its theory has a $4$-ary term $q$ satisfying the equations
$$\begin{cases}q(x_2,x_2,x_1,x_1)=x_1\\
q(x_2,x_1,x_2,x_1)=x_1\\
q(x_1,x_1,x_1,x_2)=x_1.\end{cases}$$
Such a term has been called a \emph{$3$-edge term} in~\cite{BIMMVW2010}, justifying the notation $\Edge_3$. Although $\Cube^{\Delta^*}_3\nRightarrow_\lex\Edge_3$, it was shown in~\cite{BIMMVW2010} that one can construct a $3$-edge term $q$ from a $\Delta^*$-special cube term $p$ as follows:
$$q(x,y,z,w)=p(p(y,z,z,w,w),z,p(x,z,z,z,z),w,p(x,w,z,w,z)).$$
One can indeed compute
\begin{align*}
q(x_2,x_2,x_1,x_1) &= p(p(x_2,x_1,x_1,x_1,x_1),x_1,p(x_2,x_1,x_1,x_1,x_1),x_1,p(x_2,x_1,x_1,x_1,x_1))\\
&= x_1,
\end{align*}
\begin{align*}
q(x_2,x_1,x_2,x_1) &= p(p(x_1,x_2,x_2,x_1,x_1),x_2,p(x_2,x_2,x_2,x_2,x_2),x_1,p(x_2,x_1,x_2,x_1,x_2))\\
&= p(x_1,x_2,x_2,x_1,x_1)\\
&= x_1
\end{align*}
and
\begin{align*}
q(x_1,x_1,x_1,x_2) &= p(p(x_1,x_1,x_1,x_2,x_2),x_1,p(x_1,x_1,x_1,x_1,x_1),x_2,p(x_1,x_2,x_1,x_2,x_1))\\
&= p(x_1,x_1,x_1,x_2,x_2)\\
&= x_1.
\end{align*}
However, this proof does \emph{not} extend to the partial context since in that case $p(x_2,x_1,x_1,x_1,x_1)$ in the first equation, and so $q(x_2,x_2,x_1,x_1)$, are not defined in general.
\end{remark}

Another consequence of Theorem~\ref{theorem implication} is that one can completely describe the posets of collections of categories of the form $\mclex^\strong\{\M\}$ or $\mclex_\ast^\strong\{\M\}$ for matrices $\M$ of relatively small dimensions. We consider these posets ordered by inclusion. These posets are almost isomorphic to the ones where we consider $\Mc_\all$ instead of~$\Mc_\strong$, the only difference concerns trivial matrices (see Remarks~\ref{remark empty matrices} and~\ref{remark m'=0 in theorem implication}). One can thus use a computer implementation of the algorithms recalled in Section~\ref{section preliminaries} to describe these posets. Several of these posets have been displayed in~\cite{HJ2022,HJJ2022} for the case~$\Mc_\all$. For instance, derived from Figure~1 in~\cite{HJ2022}, here is the Hasse diagram of the poset of collections of categories of the form $\mclex_\ast^\strong\{\M\}$ where $\M$ runs among the matrices in $\matr_\ast(2,3,2)$.
$$
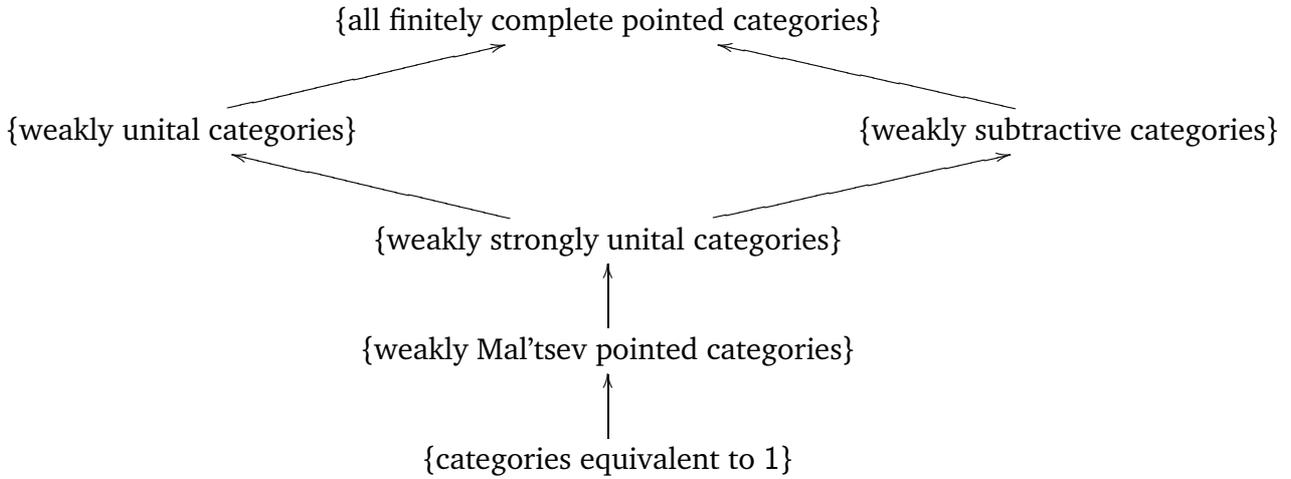
\xymatrix{&\{\text{all finitely complete pointed categories}\}& \\ \{\text{weakly unital categories}\} \ar[ru] \hspace{-39pt} && \hspace{-39pt}\{\text{weakly subtractive categories}\} \ar[lu] \\ & \{\text{weakly strongly unital categories}\} \ar[lu] \ar[ru] & \\ & \{\text{weakly Mal'tsev pointed categories}\} \ar[u] & \\ & \{\text{categories equivalent to }\1\} \ar[u] &}$$
In that diagram, a weakly unital category (respectively a weakly subtractive category, a weakly strongly unital category or a weakly Mal'tsev pointed category) is a finitely complete pointed category with $\M$-closed $\Mc_\strong$-relations where $\M=\Uni,\Sub,\StrUni,\Mal$ respectively.

To conclude, let us mention that if one takes $S=\{[x_1|x_1]\}$ in Theorem~\ref{theorem implication}, one obtains the following description of anti-trivial matrix sets. Extending the terminology from~\cite{HJ2022,HJJ2022}, we say that a matrix set $S\subseteq\matr_\ast$ is \emph{anti-trivial} if all finitely complete pointed categories have $S$-closed relations.

\begin{corollary}\label{corollary anti-trivial}
For a matrix set $S\subseteq\matr_\ast$, the following conditions are equivalent:
\begin{enumerate}[label=(\arabic*), ref=(\arabic*)]
\item\label{corollary anti-trivial pointed lex S} $S$ is anti-trivial, i.e., all finitely complete pointed categories have $S$-closed relations;
\item\label{corollary anti-trivial pointed lex} all matrices $\M$ in $S$ are anti-trivial in the sense of~\cite{HJ2022}, i.e., all finitely complete pointed categories have $\M$-closed relations;
\item\label{corollary anti-trivial pointed strong} all finitely complete pointed categories have $S$-closed $\Mc_\strong$-relations;
\item\label{corollary anti-trivial pointed essentially algebraic} all essentially algebraic pointed categories have $S$-closed $\Mc_\strong$-relations;
\item\label{corollary anti-trivial pointed general class of monos} for every finitely complete pointed category $\CC$ and every class $\Mc$ of monomorphisms in $\CC$ stable under pullbacks, $\CC$ has $S$-closed $\Mc$-relations;
\item\label{corollary anti-trivial pointed Set} $\Set_\ast$ has $S$-closed relations;
\item\label{corollary anti-trivial pointed right column} for all matrices $\M\in S$, all the entries of the right column of $\M$ are $\ast$’s or this right column of $\M$ can be found among its left columns.
\end{enumerate}
Moreover, if $S\subseteq\matr$, these are further equivalent to the following conditions:
\begin{enumerate}[resume, label=(\arabic*), ref=(\arabic*)]
\item\label{corollary anti-trivial lex S} all finitely complete categories have $S$-closed relations;
\item\label{corollary anti-trivial lex} all matrices $\M$ in $S$ are anti-trivial in the sense of~\cite{HJJ2022}, i.e., all finitely complete categories have $\M$-closed relations;
\item\label{corollary anti-trivial strong} all finitely complete categories have $S$-closed $\Mc_\strong$-relations;
\item\label{corollary anti-trivial essentially algebraic} all essentially algebraic categories have $S$-closed $\Mc_\strong$-relations;
\item\label{corollary anti-trivial general class of monos} for every finitely complete category $\CC$ and every class $\Mc$ of monomorphisms in $\CC$ stable under pullbacks, $\CC$ has $S$-closed $\Mc$-relations;
\item\label{corollary anti-trivial Set} $\Set$ has $S$-closed relations;
\item\label{corollary anti-trivial right column} for all matrices $\M\in S$, the right column of $\M$ can be found among its left columns.
\end{enumerate}
\end{corollary}

\begin{proof}
Let us first note that the equivalences \ref{corollary anti-trivial pointed lex}$\Leftrightarrow$\ref{corollary anti-trivial pointed Set}$\Leftrightarrow$\ref{corollary anti-trivial pointed right column} have been proved in~\cite{HJ2022}. Moreover, if $S\subseteq\matr$, \ref{corollary anti-trivial pointed lex}$\Leftrightarrow$\ref{corollary anti-trivial lex} has been shown also in~\cite{HJ2022} while \ref{corollary anti-trivial lex}$\Leftrightarrow$\ref{corollary anti-trivial Set}$\Leftrightarrow$\ref{corollary anti-trivial right column} were proved in~\cite{HJJ2022}. In addition, we remark that \ref{corollary anti-trivial pointed lex S}$\Leftrightarrow$\ref{corollary anti-trivial pointed lex} and \ref{corollary anti-trivial lex S}$\Leftrightarrow$\ref{corollary anti-trivial lex} if $S\subseteq\matr$ are obvious. In view of that, we can suppose without loss of generality that $S=\{\N\}$ for some $\N\in\matr_\ast$.

Now, note that if $\N\in\matr$ and $\N$ has no left columns, all statements in the theorem are obviously false (and thus equivalent), so that we can exclude this case without loss of generality. Now it suffices to apply Theorem~\ref{theorem implication} with $S'=\{[x_1|x_1]\}$; it is straightforward to show that each of its equivalent conditions becomes (equivalent to) the statements listed above. For statements~\ref{corollary anti-trivial pointed Set} (respectively~\ref{corollary anti-trivial Set}), it is easy to see that $\Part_\ast^{S'}$ (respectively $\Part^{S'}$) for that $S'$ is equivalent to $\Set_\ast$ (respectively to $\Set$) and any monomorphism is strong in those categories. Statements~\ref{corollary anti-trivial pointed general class of monos} and~\ref{corollary anti-trivial general class of monos} many seem stronger than those coming from Theorem~\ref{theorem implication} due to the conditions on~$\Mc$, but they can still be easily seen to be implied by~\ref{corollary anti-trivial pointed right column} and~\ref{corollary anti-trivial right column} respectively.
\end{proof}


\vspace{30pt}
\begin{tabular}{rl}
Email: & mhoefnagel@sun.ac.za\\
& pierre-alain.jacqmin@uclouvain.be
\end{tabular}

\end{document}